\documentclass[11pt]{article}
\usepackage[dvips]{epsfig}
\usepackage{amsmath}
\usepackage{amsfonts}
\usepackage{amssymb}
\usepackage{graphicx}
\usepackage{latexsym,amssymb,amsmath,amsthm,amsfonts} %\usepackage{amsmath,amsfonts,amssymb,mathrsfs,amscd}
\usepackage{psfrag}
\usepackage{bbm}
\usepackage{float}
\usepackage{latexsym}
\newtheorem{lem}{Lemma}[section]
\newtheorem{thm}{Theorem}[section]
\newtheorem{cor}{Corollary}[section]
\newtheorem{prop}{Proposition}[section]
\newtheorem{remark}{Remark}[section]
\newtheorem{ex}{Example}[section]
\newtheorem{defn}{Definition}[section]
\numberwithin{equation}{section}

\newcommand{\be}{\begin{equation}}
\newcommand{\ee}{\end{equation}}
\newcommand{\bi}{\bibitem}
\newcommand{\ben}{\begin{enumerate}}
\newcommand{\een}{\end{enumerate}}
\newcommand{\beq}{\begin{eqnarray}}
\newcommand{\eeq}{\end{eqnarray}}
\newcommand{\beqn}{\begin{eqnarray*}}
\newcommand{\eeqn}{\end{eqnarray*}}

\title{Almost Ricci solitons on Finsler spaces}

\author{\small{Qiaoling Xia}\footnote{Supported by the NNSFC (Nos.12071423, 12471044) and the Scientific Research Foundation of HDU (No. KYS075621060).}\\
{\small {\it Department of Mathematics, School of Sciences }}\\ {\small{\it  Hangzhou Dianzi University}}\\
 {\small {\it Hangzhou, Zhejiang Province, 310018, P.R.China}}\\
 {\small {\it E-mail address:  xiaqiaoling$@$hdu.edu.cn}}}

\date{}
\begin{document}
\maketitle{}

\begin{abstract}
In this paper, (gradient) almost Ricci solitons on Finsler measure spaces $(M, F, m)$ are introduced and investigated. We prove that $(M, F, m)$ is a gradient almost Ricci soliton with soliton scalar $\kappa$ if and only if the infinity-Ricci curvature Ric$_\infty=\kappa$ on $M$. Moreover, we give an equivalent characterization of (gradient) almost Ricci solitons for Randers metrics $F=\alpha+\beta$, which implies that every Randers (gradient) almost Ricci soliton is of isotropic S$_{BH}$-curvature. Based on this and the navigation technique, we further classify Randers almost Ricci solitons (resp., gradient almost Ricci solitons) up to classifications of Randers Einstein metrics $F$ (resp., Riemannian gradient almost Ricci solitons) and the homothetic vector fields of $F$ (resp., solutions of the equation which the weight function $f$ of $m$ satisfies) when $F$ has isotropic S$_{BH}$-curvature.  As applications, we obtain some rigidity results for compact Randers (gradient) Ricci solitons and construct several Randers gradient Ricci solitons, which are the first nontrivial examples of gradient Ricci solitons in Finsler geometry.

{\small{\it MSC 2010: }}53C60, 53C21, 53C24

{\small{\it Keywords: Finsler measure space, almost Ricci soliton, Randers metric,  conformal vector field, S-curvature.}}
\end{abstract}

\section{Introduction}\label{sec1}

 Ricci solitons, which are a generalization of Einstein metrics, were first introduced by  R. S.Hamilton (\cite{Ha2}). As self-similar solutions for Hamilton's Ricci flow (\cite{Ha1}), Ricci solitons have been studied intensively in the last twenty years (\cite{CLN}, \cite{CK}). They often arise as limits of dilations of singularities in the Ricci flow (\cite{Ha3}, \cite{Se}). In turn, the study of almost Ricci solitons, which are generalizations of quasi-Einstein manifolds related to the string theory, was started by Pigola, Rigoli, Rimoldi, Setti in \cite{PRRS}.
 Recall that {\it an almost Ricci soliton structure} $(M, g, V)$ on an $n(\geq 2)$-dimensional Riemannian manifold $(M, g)$ is the choice of a smooth vector field $V$ (if any) satisfying the soliton equation
\beq {\rm{Ric}}+\frac 12\mathcal L_Vg=\rho g,\label{R-RS}\eeq  where $\rho=\rho(x)$ is a scalar function on $M$, Ric denotes the Ricci curvature of $g$ and $\mathcal L_V$ stands for the Lie derivative in the direction of $V$. In the special case when $V=\nabla f$ for some function $f$ on $M$, we say $(M, g, f)$ is a {\it gradient almost Ricci soliton} with soliton scalar $\rho$, where $f$ is said to be a {\it potential function} of the Ricci soliton. In this case, the soliton equation becomes  \beq {\rm{Ric}}+{\rm Hess}(f)=\rho g.\label{RGS}\eeq When $\rho$ is a constant, the corresponding (gradient) almost Ricci soliton is called a {\it (gradient) Ricci soliton} with soliton constant $\rho$ (\cite{Cao}).
There is a rich literature on Riemannian (gradient) Ricci solitons (cf. \cite{CK}, \cite{CLN}, \cite{Cao} and references therein). Recently, some research progress has been made on Riemannian (gradient) almost Ricci solitons (\cite{PRRS}, \cite{BBR}, \cite{BGR}, \cite{DA}). Ricci solitons are also closely related to Einstein field theory (\cite{AW}).

 As a natural generalization of Riemannian geometry, Finsler geometry received increasing attention in recent decades. Recall that a nonnegative function $F$ on the tangent bundle  $TM$ is a {\it Finsler metric (or structure) on $M$} if (i) $F$ is smooth on $TM_0:=TM\setminus\{0\}$ and $F(x, y)=0$ if and only if $y=0$; (ii) $F(x, \lambda y)=\lambda F(x,y)$ for all $\lambda>0$; (iii)  the matrix $\left(g_{ij}(x, y)\right):=\left(\frac 12(F^2)_{y^iy^j}\right)$ is positive definite for any nonzero $y\in T_x M$. A {\it Finsler manifold} means a differential manifold $M$ equipped with a Finsler structure $F$, denoted by $(M, F)$.  Given a smooth measure $m$ on $M$, the triple $(M, F, m)$ is called a {\it Finsler measure space}. In Riemannian case, $F$ is of the form $\sqrt{h_{ij}(x)y^iy^j}$. In general, the measure $m$ on a Finsler manifold $(M, F)$ can not be uniquely determined by $F$. A Finsler measure space is not a metric space in usual sense because $F$ may be nonreversible, i.e., $F(x, y)\neq F(x, -y)$ may happen. This non-reversibility causes the asymmetry of the associated distance function and the nonlinearity for gradient or Laplacian of a function on $M$.

Let $(M, F)$ be an $n$-dimensional Finsler manifold. We say that a Finsler metric $F=F(x, y)$ is an {\it Einstein metric} with Einstein scalar $\kappa$ if its Ricci curvature
is isotropic, i.e., for any $(x, y)\in TM\setminus\{0\}$,
$$ {\rm{Ric}}(x, y)=\kappa F^2(x, y), $$ where $\kappa=\kappa(x)$ is a scalar function on $M$.
Finsler metrics for which Ricci curvatures $\kappa$ are constant (resp., zero),  are said to be {\it Ricci-constant} (resp., {\it Ricci-flat}). Note that every Einstein Riemannian metric must be Ricci-constant when $n\geq 3$ by Schur's Lemma. However this is not true in general in Finsler geometry except for some special Finsler metrics, among which Randers metrics. A well known conjecture raised by S.S. Chern's asks whether every
smooth manifold admits a Ricci-constant Finsler metric. This question has already been settled in the affirmative for dimension $2$ because, by Thurston's construction, every 2-manifold admits a complete Riemannian metric of constant Gaussian curvature (\cite{Be} for an exposition and references therein).
Much less is known about dimension greater than $2$.  To study Chern's question, D.Bao in \cite{Bao} introduced an unnormalized Ricci flow for Finsler metrics (\cite{Bao}):
 \beq \partial_t(\log F)=-F^{-2}{\rm{Ric}}, \ \ \ F|_{t=0}=F_0 \label{UNRF} \eeq inspired by Hamilton's Ricci flow in Riemannian geometry.
Both sides of (\ref{UNRF}) make sense on the spherical bundle $SM$.  It is easy to see that Ricci-flat Finsler metrics are stationary solutions of (\ref{UNRF}). There is a large class of solutions, called {\it (Finslerian) Ricci solitons}, which may be regarded as generalized fixed points for (\ref{UNRF}) (\cite{BY}). More generally, we introduce Finslerian (gradient) almost Ricci solitons, which are main research objects in present paper.
 \begin{defn} Let $(M, F)$ be an $n$-dimensional Finsler manifold. $(M, F, V)$ is called an almost Ricci soliton with soliton scalar $\kappa$ if there is a $C^2$ vector field $V$ on $M$ such that $(F, V)$ satisfies
\beq 2{\rm{Ric}}+{\mathcal L}_{\hat V}(F^2)=2\kappa F^2, \label{RS-kappa}\eeq for some function $\kappa=\kappa(x)$ on $M$, where $\mathcal L_{\hat V}$ is the Lie derivative in the direction $\hat V$, which is the complete lift of $V$ on $TM$. In particular, if $\kappa$ is a constant,  $(M, F, V)$ is called a Ricci soliton with soliton constant $\kappa$.\end{defn}

Obviously, when $V=0$ or $V$ is a Killing vector field of $F$, the solutions  of (\ref{RS-kappa}) are exactly Einstein metrics on $M$, called {\it trivial almost Ricci solitons}. For any $(x, y)\in TM$, if $y=0$, (\ref{RS-kappa}) trivially holds since the both sides of (\ref{RS-kappa}) are identically zero. if $y\in T_xM\setminus\{0\}$, there is an open neighborhood $U_x\subset M$ and a non-vanishing geodesic field $Y$ on $U_x$ such that $Y_x=y$. $Y$ is called a {\it geodesic extension} of $y$. This geodesic field $Y$ induces a Riemannian metric $g_Y=(g_{ij}(Y))$ on $U_x$, where $g_{ij} (1\leq i, j\leq n)$ are the fundamental tensors of $F$. We denote by $dV_{g_Y}$ the Riemannian volume form of $g_{Y}$. Given a smooth measure $m$ on $M$, we decompose the volume form $dm=e^{-\psi}dV_{g_Y}$ in the direction of geodesic field $Y$ on $U_x$, where $\psi$ is a smooth function on $U_x$. Note that $\psi$ depends on the choice of $Y$.
 \begin{defn}  With notations as above, if there is a vector field $V$ on $M$ such that $V_x=({\rm grad}_{g_Y}\psi)(x)$ (the gradient of $\psi$ with respect to $g_Y$ at $x$) and $(F, V)$ satisfies (\ref{RS-kappa}) at $(x, y)\in TM$, then $(M, F, m)$ is called a {\it gradient almost Ricci soliton} with soliton scalar $\kappa$. In particular, it is called a {\it gradient Ricci soliton} with soliton constant $\kappa$ if $\kappa$ is a constant. \end{defn}

  When $F=g$ is Riemannian, (gradient) almost Ricci solitons are reduced to the corresponding ones in Riemannian geometry. It is known that a gradient Ricci soliton is a Ricci soliton in Riemannian case. However it is not true in general in Finslerian case because the Finsler metric $F$ depends on the tangent direction $y$ as well as the point $x$. A (gradient) almost Ricci soliton is said to be {\it expanding, steady or shrinking} respectively if $\kappa<0$, $\kappa=0$ or $\kappa>0$. Otherwise, it is said to be {\it indefinite}. We say that a (gradient) almost Ricci soliton is {\it forward (resp., backward) complete} if $(M, F)$ is forward (resp., backward) complete. It is said to be {\it complete} if it is both forward complete and backward complete. If $M$ is compact, then a (gradient) almost Ricci soliton is said to be a {\it compact (gradient) almost Ricci soliton}.  Gradient Ricci solitons are related with the infinity-Ricci curvature Ric$_\infty$ introduced by S.Ohta (\cite{Oh1}). The following result gives a geometric explanation of gradient Ricci solitons.
\begin{thm} \label{thm11} Let $(M, F, m)$ be an $n$-dimensional Finsler measure space. Then $(M, F, m)$ is a gradient almost Ricci soliton with soliton scalar $\kappa$ if and only if Ric$_\infty=\kappa$ with respect to $m$, where $\kappa=\kappa(x)$ is a scalar function on $M$. In particular, $(M, F, m)$ is a gradient Ricci soliton with soliton constant $\kappa$ if and only if Ric$_\infty=\kappa$ with respect to $m$.
\end{thm}

 If $F=g$ is Riemannian, then Theorem \ref{thm11} implies that $(M, g, f)$ is a gradient Ricci soliton if and only if the $\infty$-dimensional Bakry-Emery Ricci curvature on the weighted Riemannian manifold $(M, g, e^{-f}dV_g)$ is constant (\cite{Li}).
Because of the nonlinearity of (\ref{UNRF}) in $y$,  the existence and convergence of the solutions to (\ref{UNRF}) is open up to now. Although there are generalized solutions (i.e., Ricci solitons) to (\ref{UNRF}), there are few results including nontrivial examples for non-Riemannian Finslerian Ricci solitons except for \cite{MZZ}. This is one of the motivations to write this paper. Some nontrivial examples are very important to further study Finslerian unnormalized Ricci flow (resp., Ricci soliton). Let us consider Randers metrics for simplicity.

Randers metrics are important class of Finsler metrics, which are expressed by $F=\alpha+\beta$, where $\alpha^2=a_{ij}(x)y^iy^j$ is a Riemannian metric and $\beta=b_i(x)y^i$ is a 1-form with $b:=\|\beta\|_\alpha<1$ on an $n$-dimensional differential manifold $M$. They were originated from the research on
the general relativity (\cite{Ra}) and have been widely applied in biology, physics and psychology, etc. (\cite{AIM}). In particular, Randers metrics are the solution of Zermelo's navigation problem (\cite{BRS}). To state our results, we introduce some notations. We denote by $b_{i;j}$ the covariant derivative of $\beta$ with respect to the Levi-Civita connection of $\alpha$ and use $a_{ij}$ to raise and
lower the indices of $b_i$, $r_{ij}$, $s_{ij}$, $r_i$, $s_i$ and $y^i$ etc., throughout the paper.  Let
\beqn & r_{ij}:=\frac 12(b_{i;j}+b_{j;i}), \ \  r_j: =b^i r_{ij}, \ \ s_{ij}:=\frac 12(b_{i;j}-b_{j;i}), \ \  s^i_{\ j}:=a^{ik}s_{kj}, \\
& s_j:=b_is^i_{\ j}=b^is_{ij}, \ \ e_{ij}=r_{ij}+b_i s_j+b_j s_i,\ \  t_{ij}:= s_{ik}s^k_{\ j},\ \ \ t^i_{\ j}:=a^{ik}t_{kj}, \ \ \ t_j:= b^it_{ij}. \eeqn In the following, we denote by $r_0$, $s_0$ and $r_{00}$, $e_{00}$ the contractions with $y^i$ once and twice respectively, i.e.,  $r_0:=r_jy^j$, $s_0:=s_jy^j$, and $r_{00}:=r_{ij}y^iy^j$, $e_{00}:=e_{ij}y^iy^j$.  Moreover, we always denote by $V$ a $C^2$ vector field on $M$ throughout the paper unless otherwise stated.

\begin{thm} \label{thm12} Let $F=\alpha+\beta$ be a non-Riemannian Randers metric on an $n (\geq 2)$-dimensional manifold $M$ and $V$ be a vector field on $M$. Then $(M, F, V)$ is an almost Ricci soliton with soliton scalar $\kappa$ if and only if there are scalar functions $c=c(x)$ and $\sigma=\sigma(x)$ on $M$ such that $V$ is a conformal vector field of $\alpha$ with conformal factor $c$, i.e., $V_{i;j}+V_{j;i}=4ca_{ij}$,  and the following identities hold on $TM$:
\beq e_{00}&=& 2\sigma(\alpha^2-\beta^2), \label{e00-1}\\
 {}^{\alpha}{\rm{Ric}}&=& (\kappa-2c)(\alpha^2+\beta^2)+t^i_{\ i}\alpha^2+2t_{00}-(n-1)\sigma^2(3\alpha^2-\beta^2)\nonumber \\
& &+2(n-1)\sigma_0\beta-(n-1)(s_0^2+s_{0;0}).\label{a-Ric-1} \eeq
where ${}^{\alpha}{\rm{Ric}}$ is the Ricci curvature of $\alpha$ and $\sigma_0:=\sigma_{x^i}y^i$ with
\beq 3(n-1)\sigma_0&=& 2c\beta-{\mathcal L_{\hat V}(\beta)}, \label{tau-0-1}\eeq  here ${\mathcal L}_{\hat V}(\beta)$ is the Lie derivative of $\beta$  with respect to the complete lift $\hat V$ of $V$ given by (\ref{L-ab})$_2$.
\end{thm}

Note that (\ref{e00-1}) is equivalent to that $F$ is of isotropic S$_{BH}$-curvature $\sigma$, i.e., $S_{BH}=(n+1)\sigma F$, where S$_{BH}$ stands for the S-curvature of $F$ with respect to the Busemann-Hausdorff measure $m_{BH}$ (see Lemma \ref{lem-41} below). On the other hand, since $(M, F, V)$ is a non-Riemannian almost Ricci soliton, then $\mathcal L_{\hat V}(\alpha)=2c\alpha$ by Theorem \ref{thm12}. In this case, if $V$ further satisfies $\mathcal L_{\hat V}(\beta)=2c\beta$, then $V$ is a conformal vector field of $F$ with conformal factor $c$ by Proposition \ref{prop22} below. Consequently, $V$ is homothetic by Theorem 1.1 in \cite{HM} (also see Theorem 1.2, \cite{SX}). Moreover, $\sigma$ is constant by (\ref{tau-0-1}) when $n\geq 2$. Thus one obtains
\begin{cor}\label{cor11} Let $F=\alpha+\beta$ be a non-Riemannian Randers metric on an $n(\geq 2)$-dimensional manifold $M$.  If there is a vector field $V$ on $M$ such that $(M, F, V)$ is an almost Ricci soliton,  then $F$ is of isotropic S$_{BH}$-curvature and $V$ is a conformal vector field of $\alpha$.

Further,  if $V$ satisfies  $\mathcal L_{\hat V}(\beta)=2c\beta$,  then $F$ is of constant S$_{BH}$-curvature and $V$ is a homothetic vector field of $F$.
\end{cor}

  In particular, when $V=0$, which implies that $c=0$,  almost Ricci solitons are reduced to Einstein metrics, which are  trivial Ricci solitons, and (\ref{tau-0-1}) implies that $\sigma$ is constant when $n\geq 2$. Thus we obtain an equivalent characterization for Einstein metrics, which was due to D.Bao and C. Robles (\cite{BR}).
\begin{cor} \label{cor12} (\cite{BR}) Let $F=\alpha+\beta$ be a Randers metric on an $n(\geq 2)$-dimensional manifold $M$. Then $(M, F)$ is an Einstein metric with Einstein scalar $\kappa$ if and only if $F$ is of constant S$_{BH}$-curvature $\sigma$ and
\beq {}^{\alpha}{\rm{Ric}}=\kappa(\alpha^2+\beta^2)+ t^i_{\ i}\alpha^2+2t_{00}-(n-1)\sigma^2(3\alpha^2-\beta^2)-(n-1)(s_0^2+s_{0;0}).\label{a-Ric-2} \eeq
\end{cor}

On the other hand,  Randers metrics $F=\alpha+\beta$ are solutions to the navigation problem (\cite{ChS}). Let $h^2=h_{ij}(x)y^iy^j$ be a Riemannian metric and $W=W^i(x)\frac{\partial}{\partial x^i}$ be a vector field with $\|W\|_h<1$ on $M$.
Then the solution to the navigation problem is a geodesic of the metric $F$ satisfying
$$h\left(x, \frac y{F(x, y)}-W_x\right)=1.$$ By solving this equation, one obtains a Randers metric $F$ given by
\beq F=\frac{\sqrt{\lambda h^2+W_0^2}}{\lambda}-\frac{W_0}{\lambda}, \ \ \ \  \ W_0=W_iy^i,\label{F-hW} \eeq where $W_i:=h_{ij}W^j$ and $\lambda: =1-\|W\|_h^2$. Conversely, every Randers metric $F=\alpha+\beta$ on a manifold $M$ can be obtained from a Riemann metric $h$ and a vector field $W$ with $\|W\|_h<1$ on $M$ (\cite{ChS}, \cite{BRS}). We call (\ref{F-hW}) the {\it navigation representation} of $F$ and  $(h, W)$ the {\it navigation data} of $F$. $F$ is Riemannian if and only if $W=0$,  and $F$ is of isotropic (resp., constant) S$_{BH}$-curvature $\sigma$ if and only if $W$ is a conformal (resp., homothetic) vector field of $h$ with conformal factor $-\sigma$ (\cite{Xing} or Lemma \ref{lem51} in \S 5). Based on this and Corollary \ref{cor11}, we obtain the following result. In particular, if $(M, F, V)$ is an almost Ricci soliton, then $(M, h)$ must be Einstein.

\begin{thm} \label{thm13*} Let $F$ be a non-Riemannian Randers metric with navigation data $(h, W)$ on an $n (\geq 2)$-dimensional manifold $M$ and $V$ be a vector field on $M$. Then  $(M, F, V)$ is an almost Ricci soliton with soliton scalar $\kappa$ if and only if there are scalar functions $\mu=\mu(x)$ and $\sigma=\sigma(x)$ on $M$ such that $h$ is an Einstein metric with Einstein scalar $\mu$, $W$ is a conformal vector field of $h$ with conformal factor $-\sigma$ and $V$ satisfies
\beq \mathcal L_{\hat V}(h^2)&=&2ch^2-6(n-1)\left\{(\sigma_iW^i)h^2+\sigma_0W_0\right\},\label{LV-h2} \\
 \mathcal L_{\hat V}(W_0)&=&c W_0-3(n-1)\left\{2(\sigma_iW^i)W_0-\lambda\sigma_0\right\},\label{LVW0}\eeq where $c:=\kappa-\mu+(n-1)\sigma^2+2(n-1)\sigma_iW^i$ and $\hat V$ is the complete lift of $V$ on $TM$. \end{thm}

When $V=0$,  Randers almost Ricci solitons are just Einstein metrics (i.e., trivial Ricci solitons). In this case, (\ref{LV-h2}) implies that $\sigma$ is constant by the irreducibility of $h^2$ and hence $c=0$. Note that every Riemannian Einstein metric must be Ricci-constant when $n\geq 3$ by Schur's Lemma. Thus Theorem \ref{thm13*} is reduced to Bao-Robles' result, which actually classify Einstein Randers metrics up to the classifications of Einstein Riemannian metrics $h$ and of the homothetic vector fields for $h$.
\begin{cor} \label{cor13} (\cite{BR}) Let $F$ be a non-Riemannian Randers metric with navigation data $(h, W)$ on an $n(\geq 2)$-dimensional manifold $M$. Then $(M, F)$ is an Einstein metric with Einstein scalar $\kappa$ if and only if $h$ is an Einstein metric with Einstein scalar $\mu$ and $W$ is a homothetic vector field of $h$ with dilation $-\sigma$, where $\mu=\kappa+(n-1)\sigma^2$. In particular, $\kappa$, $\mu$ must be constants when $n\geq 3$.
\end{cor}

Similar to Corollary \ref{cor11}, one obtains the following result from Theorem \ref{thm13*}. Its proof will be given at the end of \S \ref{sec5}.
\begin{cor} \label{cor13*}  Let $F$ be a non-Riemannian Randers metric with navigation data $(h, W)$ on an $n(\geq 2)$-dimensional manifold $M$. If there is a vector field $V$ such that $(M, F, V)$ is an almost Ricci soliton, then $h$ is Einstein and $W$ is a conformal vector field of $h$. In this case, $F$ is of isotropic S$_{BH}$-curvature.

Further, if $h$ satisfies $\mathcal L_{\hat V}(h^2)=2ch^2$ or $W$ satisfies $\mathcal L_{\hat V}(W_0)=c W_0$ with $\|W\|_h^2\neq \frac 13$, where $c$ is defined as in Theorem \ref{thm13*}, then $h$ is Einstein, $W$ is a homothetic vector field of $h$ with dilation $-\sigma$ and $V$ is a homothetic vector field of $F$ with dilation $\frac c2$. In this case, $F$ is of constant S$_{BH}$-curvature.
 % and $(M, F, V)$ is a Ricci soliton with soliton constant $\mu+c-(n-1)\sigma^2$.
\end{cor}

Corollaries \ref{cor11} and \ref{cor13*} respectively give a sufficient condition such that $F$ is of constant S$_{BH}$-curvature. Under this condition, we can classify Randers almost Ricci solitons (resp., Ricci solitons) $(M, F, V)$ up to the classifications of Randers Einstein metrics $F$ and homothetic vector fields for $F$ by Theorem \ref{thm13*} and Corollaries \ref{cor13}-\ref{cor13*}. More precisely, we have
\begin{cor} \label{cor15} Let $F$ be a non-Riemannian Randers metric with navigation data $(h, W)$ on an $n(\geq 2)$-dimensional manifold $M$ and $V$ be a vector field on $M$. If $F$ is of constant S$_{BH}$-curvature $\sigma$, then $(M, F, V)$ is an almost Ricci soliton with soliton scalar $\kappa$ if and only if one of the following statements holds.

 (1) $h$ is Einstein with Einstein scalar $\mu$, $W$ is a homothetic vector field of $h$ with dilation $-\sigma$ and $V$ is a homothetic vector field of $F$ with dilation $\frac 12c$, where $\kappa=\mu+c-(n-1)\sigma^2$. In particular, both $\kappa$ and $\mu$ are constants when $n\geq 3$.

 (2) $F$ is Einstein with Einstein scalar $\nu$ and $V$ is a homothetic vector field of $F$ with dilation $\frac 12c$, where $\kappa=\nu+c$. In particular, both $\kappa$ and $\nu$ are constants when $n\geq 3$.
\end{cor}

Conformal vector fields on a Randers manifold were studied in \cite{SX}. In particular, the authors give a classification of conformal (especially, homothetic) vector fields on an $n(\geq 3)$-dimensional Randers manifold of constant flag curvature. Next we further consider the characterization of Randers gradient almost Ricci solitons.
% When $n=3$,  Einstein Randers manifolds are those of constant flag curvature by Schur's Lemma. Thus we can obtain all Randers almost Ricci solitons when $n=3$ and $F$ is of constant S$_{BH}$-curvature.  It is a natural question to classify conformal (esp. homothetic) vector fields on an $n(>3)$-dimensional Einstein Randers manifold. We will argue this question somewhere else.

Given a smooth measure $m$ on a Randers manifold $(M, F)$, the volume form $dm$ determined by $m$ can be written as $dm=e^{-f}dm_{BH}$ for some smooth function $f$ on $M$, where $dm_{BH}$ is the volume form on $M$ determined by the Busemann-Hausdorff measure $m_{BH}$. In this case, $f$ is called the {\it weighted function} of $m$, here we regard $m_{BH}$ as a standard measure on $(M, F)$.  When $F=g$ is Riemannian, $dm=e^{-f}dV_g$, which is exactly the weighted measure on $(M, g)$.
\begin{thm}\label{thm14} Let $F=\alpha+\beta$ be a Randers metric on an $n$-dimensional measure space $(M, m)$ and $f$ be the weighted function of $m$. Then $(M, F, m)$ is a gradient almost Ricci soliton with soliton scalar $\kappa$ if and only if there is a function $\sigma$ on $M$ such that
\beq e_{00}&=&2\sigma(\alpha^2-\beta^2),\label{GRS-e00}\\
{}^{\alpha}{\rm Ric}&=&\kappa(\alpha^2+\beta^2)+2t_{00}+t^i_{\ i}\alpha^2-2n\sigma_0\beta\nonumber \\
 & &-(n-1)\left(s_0^2+s_{0;0}+3\sigma^2\alpha^2-\sigma^2\beta^2\right)-2(s_0+\sigma\beta)f_{0}-{\rm Hess}_\alpha(f)(y),\label{GRS-Ric-a}\eeq with
\beq (2n-1)(1-b^2)\sigma_0=\sigma(1+b^2)f_{0}+f_i(s^i_{\ 0}-s^i\beta)+f_{;0j}b^j+(s_0+2\sigma\beta)(f_{i}b^i),\label{sigma0-f}\eeq where $f_{k}, f_{;ij}$ are respectively the first, second order covariant derivatives of $f$  and ${\rm Hess}_\alpha(f)$ is the Hessian of $f$ with respect to $\alpha$. In this case,  $S_{BH}=(n+1)\sigma F$ . In particular, $\sigma$ is constant if $f$ satisfies
\beq \sigma(1+b^2)f_{0}+f_i(s^i_{\ 0}-s^i\beta)+f_{;0j}b^j+(s_0+2\sigma\beta)(f_{i}b^i)=0.\label{f-eq}\eeq
 \end{thm}

When $\kappa$ is constant, we obtain an equivalent characterization for Randers gradient Ricci solitons from Theorem \ref{thm14}. Recently, Mo-Zhu-Zhu also gave an equivalent characterization for Randers gradient Ricci solitons (\cite{MZZ}, Theorem 1.1). In fact, (1.3)-(1.5) in  Theorem 1.1 of \cite{MZZ} essentially coincide with (\ref{GRS-e00}) and (\ref{G-si0i})-(\ref{G-Ric-a}) in Theorem \ref{thm61} (see Remark \ref{rm61}). However, the proof of Theorem \ref{thm61} is simpler. In present paper, we further show that (1.4) in \cite{MZZ} (i.e., (\ref{G-si0i}) in Theorem \ref{thm61}) is equivalent to (\ref{sigma0-f}) under the assumptions (\ref{GRS-e00})-(\ref{GRS-Ric-a}) (see Lemma \ref{lem62} in \S 6). Thus Theorem \ref{thm14} is more elaborate than Theorem 1.1 in \cite{MZZ} and  (\ref{sigma0-f}) gives a convenient criteria to judge when $F$ is of constant S$_{BH}$-curvature. Obviously,  (\ref{sigma0-f}) implies that $\sigma$ is constant when $f$ is constant. In this case, $\dot S\equiv 0$ and Randers gradient almost Ricci solitons are reduced to Einstein metrics (i.e., trivial gradient almost Ricci solitons) and Theorem \ref{thm14} is reduced to Corollary \ref{cor12}. In the following we give a navigation description for Randers gradient almost Ricci solitons.  We denote by $W_{i:j}$ the covariant derivative of $W$ with respect to the Levi-Civita connection of $h$ and let
\beq \mathcal R_{ij}:=\frac 12 (W_{i:j}+W_{j:i}), \ \ \ \ \mathcal S_{ij}:=\frac 12 (W_{i:j}-W_{j:i}).\label{RS} \eeq We use $h_{ij}$ to raise and
lower the indices of $W_i$, $\mathcal R_{ij}$, $\mathcal S_{ij}$. We denote by $W_0$ and $\mathcal S^i_{\ 0}$  the contractions of $W_i$ and $\mathcal S^j_{\ i}$ with $y^i$ respectively. Similarly, for any function $f$ on $M$, we also denote by $f_0$ the contraction of the ordinary derivative $f_i: =f_{x^i}$ with $y^i$ throughout the paper.

\begin{thm} \label{thm15} Let $F$ be a Randers metric with navigation data $(h, W)$ on an $n$-dimensional measure space $(M, m)$ and $f$ be the weighted function of $m$. Then $(M, F, m)$ is a gradient almost Ricci soliton with soliton scalar $\kappa$ if and only if there are scalar functions $\mu$ and $\sigma$ on $M$ such that $(M, h, f)$ is a Riemannian gradient almost Ricci soliton with  soliton scalar $\mu$, $W$ is a conformal vector field of $h$ with conformal factor $-\sigma$ and $f$ satisfies
\beq & (2n-1)\sigma_0 =\sigma  f_0-f_k{\mathcal S}^k_{\ 0}- f_{:0j}W^j,\label{sfSWx*}\\
& (\sigma_i-\sigma f_i)W^i=\kappa-\mu+(n-1)\sigma^2. \label{H-hfW} \eeq   In this case, $S_{BH}=(n+1)\sigma F$. In particular, $\sigma$ is constant if
$\sigma  f_0-f_k{\mathcal S}^k_{\ 0}- f_{:0j}W^j=0.$
\end{thm}

Theorem \ref{thm15} shows that the underlying Riemannian structure $(M, h, f)$ is itself a gradient almost Ricci soliton if $(M, F, m)$ is a gradient almost Ricci soliton. Recall that $F$ is of constant S$_{BH}$-curvature $\sigma$ if and only if $W$ is a homothetic vector field of $h$ with dilation $-\sigma$. Hence, we have
\begin{cor}\label{cor16}  Under the same assumptions as in Theorem \ref{thm15},  if $F$ is of constant S$_{BH}$-curvature $\sigma$, then $(M, F, m)$ is a gradient almost Ricci soliton with soliton scalar $\kappa$ if and only if $(M, h, f)$ is a Riemannian gradient almost Ricci soliton with  soliton scalar $\mu$ and $f$ satisfies
\beq \sigma  f_0-f_k{\mathcal S}^k_{\ 0}- f_{:0j}W^j=0,\label{f-sigma-hw}\eeq
 where $\mu:=\kappa+(n-1)\sigma^2+\sigma f_iW^i$.

In particular, if $df(W)=0$ (resp., $\sigma=0$), then $(M, F, m)$ is a gradient Ricci soliton with soliton constant $\kappa$ if and only if $(M, h, f)$ is a Riemannian gradient Ricci soliton with soliton constant $\kappa+(n-1)\sigma^2$ (resp., soliton constant $\kappa$),  and $f$ satisfies (\ref{f-sigma-hw}) (resp., $f$ satisfies $f_k{\mathcal S}^k_{\ 0}+f_{:0j}W^j=0$).
\end{cor}

In Riemannian case, every compact gradient steady or expanding Ricci solitons $(M, g)$ is Ricci-constant, i.e., $g$ is Einstein with Ricci constant (\cite{Ha3}). In this case, the function $f$ is necessarily constant.
On the other hand, Perelman's results claim that any compact Riemannnian Ricci soliton is necessarily a gradient Ricci soliton (\cite{Per}). Thus every compact Riemannian steady or expanding Ricci soliton must be Ricci-constant. For shrinking Ricci solions, any two or three dimensional compact shrinking Ricci soliton must be Ricci-constant (\cite{Ha3}, \cite{Iv}). However, in general this is not true. In fact, $\mathbb{CP}^2\# (-\mathbb {CP}^2)$  is a compact non-Einstein shrinking Ricci soliton (\cite{Cao}). From these, Corollaries \ref{cor15}-\ref{cor16} and Corollary \ref{cor13}, one obtains the following rigidity result.
\begin{thm} \label{thm16} Let $F$ be a non-Riemannian Randers metric of constant S$_{BH}$-curvature $\sigma$ on an $n (\geq 2)$-dimensional compact manifold $M$.

(1) If there is a vector field $V$ on $M$ such that $(M, F, V)$ is a Ricci soliton, then $F$ is Ricci-constant and $V$ is a homothetic field of $F$.

(2)  If there is a measure $m$ on $(M, F)$ with the weighted function $f$ such that $(M, F, m)$ is a gradient steady or expanding Ricci soliton and $\sigma=0$, then $F$ is Ricci-constant and $f$ is constant, i.e., $m=m_{BH}$ up to a positive constant. \end{thm}

 Based on Corollary \ref{cor16}, we shall construct shrinking, steady and expanding gradient Ricci solitons for Randers metrics in \S \ref{sec7}, which are the first nontrivial examples of gradient Ricci solitons in Finsler geometry.

\section{Preliminaries} \label{sec2}
In this section, we review some basic concepts and notations in Finsler geometry. We refer to \cite{ChS} and \cite{Sh2} for more details.
\subsection{Connection and curvatures} \label{subsec21}

Let $(M, F)$ be an $n$-dimensional Finsler manifold $M$ and $TM$ ($T^*M$) the tangent (cotangent) bundle of $M$. The geodesic coefficients $G^i$ of $F$ are defined by
\beqn G^i:=\frac 14 g^{il}\{[F^2]_{x^my^l}y^m-[F^2]_{x^l}\},\eeqn where $(g^{ij}):=(g_{ij})^{-1}$, here $g_{ij}:=\frac 12[F^2]_{y^iy^j}$ are the fundamental tensors of $F$.

Let $\pi: TM_0\rightarrow M$ be a projection from the slit tangent bundle $TM_0:=TM\setminus\{0\}$ to $M$.  Denote by $\pi^*TM$ the pull-back tangent bundle and $\pi^*T^*M$ the pull-back cotangent bundle over $TM_0$. Then $\pi^*TM$ admits a unique linear connection, called {\it Chern connection} $D$, which is torsion free and almost compatible with the metric $F$, that is, the connection coefficients $\Gamma_{ij}^k$ of $D$ satisfy  $\Gamma_{ij}^k=\Gamma_{ji}^k$ and
\beq \frac{\partial g_{ij}}{\partial x^k}=g_{il}\Gamma_{jk}^l+g_{lj}\Gamma_{ik}^l+2C_{ijl}N^l_{\ k}, \ \ \ {\rm equivalently}, \ \ \ g_{ij|k}=0,  \label{g-com}\eeq where  $C_{ijk}:=\frac 14[F^2]_{y^iy^jy^k}$ is the Cartan tensors of $F$ and $``|" $ means horizontal covariant derivative with respect to the Chern connection of $F$ (\cite{Sh2}).
A smooth curve $\gamma: [a, b]\rightarrow M$ is called a {\it geodesic} on $(M, F)$ if $\gamma=(\gamma^i)$ satisfies the equation
\beq \ddot \gamma^i(t)+2G^i(\gamma, \dot\gamma)=0\label{geo-eq}\eeq in local coordinates.

For any $x\in M$ and $y\in T_x(M)\setminus\{0\}$, the {\it Riemann curvature} $R_y={R^i}_k(x, y)\frac{\partial}{\partial x^i}\otimes dx^k$ of $F$ is defined by
\begin{eqnarray}{R^i}_{k}=2\frac{\partial G^i}{\partial x^k}-\frac{\partial^2G^i}{\partial x^j\partial y^k}y^j+2G^j\frac{\partial^2 G^i}{\partial y^j\partial
y^k}-\frac{\partial G^i}{\partial y^j}\frac{\partial G^j}{\partial y^k}.\label{Rik1}\end{eqnarray}
The {\it Ricci curvature} of $F$ is defined by ${\rm{Ric}}:=R^i_{\ i}.$  $F$ is said to be an {\it  Einstein metric} with Einstein scalar $\kappa$ if there is a scalar function $\kappa=\kappa(x)$  on $M$ such that
${\rm{Ric}}=\kappa F^2.$ When $F=h$ is Riemannian on $M$, $h$ is an Einstein metric with Ricci scalar $\kappa$ if $^h{\rm Ric}=\kappa h^2$. If $\kappa$ is constant (resp., zero), then $F$ or $h$ is said to be {\it Ricci-constant} (resp., {\it Ricci-flat}).

 A vector field $Y$ on an open subset $U\subset M$ is called a {\it geodesic field} if every integral curve of $Y$ in $U$ is a geodesic of $F$. For any $y\in T_xM\setminus\{0\}$, there are an open neighborhood $U_x$ and a non-vanishing $C^\infty$ geodesic field $Y$ on $U_x$ such that $Y$ is a geodesic extension of $y$. In local coordinates, a geodesic field $Y=Y^i\frac{\partial}{\partial x^i}$ is characterized by
 \beqn Y^j(x)\frac{\partial Y^i}{\partial x^j}(x)+2G^i(x, Y_x)=0,\label{Y-eq}\eeqn  here we identify $Y_x=Y^i(x)\frac{\partial}{\partial x^i}|_x$ with $(Y^1(x), \cdots, Y^n(x))$.
 The geodesic field $Y$ on $U$ induces a Riemannian metric $\hat g:=g_Y$ on $U$. Then $Y$ is also a geodesic field of $\hat g$. Moreover,  $G^i(Y)=\hat G^i(Y)$ and $R_y=\hat R_y$, where $\hat G^i$ and $\hat R_y$ are respectively the geodesic coefficients and the Riemann curvature of $\hat g$. Since the Ricci curvature Ric is the trace of the Riemann curvature $R_y$, the Ricci curvatures of $F$ and $\hat g$ coincide, i.e.,
 \beq {\rm Ric}(x, y)=\widehat{\rm Ric}(x, y) \label{Ric-Ric}\eeq for every $(x, y)\in TU_0$ and geodesic extension $Y$ of $y$  (see \S 6.2, \cite{Sh2}).

 Given a smooth measure $m$, write $dm=\sigma_F(x)dx$, where $\sigma_F$ is called the {\it density function} of $m$.  The {\it distortion} of $F$ is defined by
\beq \tau(x, y):=\log\left\{ \frac{\sqrt{det(g_{ij}(x,
y))}}{\sigma_F(x)}\right\}.\label{tau}\eeq  For any $y\in T_{x}M\backslash\{0\}$, let $\eta(t)$ be a geodesic with $\eta(0)=x$ and $\dot{\eta}(0)=y$.
The {\it S-curvature} $S$ is defined as the change rate of the distortion $\tau$ along $\eta$, i.e., $$S(x, y)=\frac d{dt}\tau(\eta(t), \dot\eta(t))|_{t=0}.$$ By a direct calculation, we have
\beq S(x, y)=\tau_{|i}(x, y)y^i=\frac{\partial G^i}{\partial y^i}-y^i\frac{\partial}{\partial x^i}\left(\log \sigma_F\right). \label{S-curv}\eeq
 S-curvature vanishes when $dm$ is the Riemannian measure of a Riemannian manifold $(M, g)$. If $dm=e^{-\psi}\sqrt{{\rm det}g_{ij}}dx$ is a weighted volume measure of $(M, g)$, where $\psi$ is a smooth function on $M$,  then $\tau=\psi$ and $S=d\psi(y)$.
$F$ is said to be of {\it isotropic S-curvature $\sigma$} if there exists a function $\sigma=\sigma(x)$ on $M$ such that $S(x, y)=(n+1)\sigma F(x, y)$ for any $(x, y)\in TM$. In particular, if $\sigma$ is constant, we say that $F$ is of {\it constant S-curvature $\sigma$} (\cite{ChS}, \cite{CS2}).

In Finsler geometry, there is a frequently used volume measure called the {\it Busemann-Hausdorff  measure} $m_{BH}$, whose volume form $dm_{BH}=\sigma_{BH}dx$, where
\beqn \sigma_{BH}=\frac{\mbox{Vol}(\mathbb B_1^n(x))}{\mbox{Vol}\left\{y=(y^i)\in \mathbb R^n|F(x, y)<1\right\}}, \eeqn where $\mathbb B^n_1$ is the unit ball in $\mathbb R^n$ and Vol means the Euclidean volume. This volume form can be expressed explicitly in some special cases, such as Randers metrics etc.(see \S \ref{sec4} below). When $F=g$ is Riemannian, $m_{BH}$ is reduced to the Riemannian measure of $g$, equivalently, $dm_{BH}=dV_g$. If $F$ is of  isotropic (resp., constant) S-curvature $\sigma$ with respect to the Busemann-Hausdorff measure $m_{BH}$, denoted by $S_{BH}=(n+1)\sigma F$, then we say $F$ has {\it isotropic} (resp., {\it constant}) S$_{BH}$-curvature $\sigma$.

Inspired from the weighted Ricci curvature on Riemannian manifolds, S. Ohta in \cite{Oh1} introduced the weighted Ricci curvature Ric$_N$ of $(M, F, m)$ in terms of Ricci curvature Ric  and S-curvature. More precisely, we have
\begin{defn}\label{def21}(\cite{Oh1})
Given a nonzero vector $y\in T_{x}M$, let $\eta:(-\varepsilon,\varepsilon)\rightarrow M$ be the geodesic with $\eta(0)=x$ and $\dot\eta(0)=y$. We set $dm =e^{-\psi(\eta(t))}dV_{\dot\eta}$ along $\eta$, where $dV_{\dot\eta}$ is the volume form of $g_{\dot\eta}$. Define the weighted  Ricci curvature involving a parameter $N\in(n,\infty)$ by
 $${\rm{Ric}}_N(x, y):={\rm{Ric}}(x, y)+{(\psi\circ \eta)^{''}(0)}-\frac{(\psi\circ \eta)^{\prime}(0)^{2}}{(N-n)},$$ where {\rm Ric} is the Ricci curvature of $F$. Further, define ${\rm{Ric}}_n: =\lim\limits_{N\rightarrow n}{\rm Ric}_N$ and ${\rm Ric}_\infty:=\lim\limits_{N\rightarrow \infty}{\rm Ric}_N$.
% For any $\lambda\geq 0$ and $N\in[n, \infty]$, define ${\rm{Ric}}_N(\lambda y):=\lambda^2{\rm{Ric}}_N(y)$.
\end{defn}

It is easy to see that $(\psi\circ \eta)^{'}(0)=S(x, y)$ and $(\psi\circ \eta)^{''}(0)$ is exactly the change rate of the $S$-curvature along the geodesic $\eta$, denoted by $\dot S$. Thus,
$${\rm Ric}_\infty=\textmd{Ric}+\dot S. $$
 For any $N\in [n, \infty]$, we say that  $Ric_N\geq K (K\in \mathbb R)$,  if ${\rm{Ric}}_N(x, y)\geq KF^2(x, y)$ for all $(x, y)\in TM$. Ohta proved in \cite{Oh1} that the bound $\textmd{Ric}_N\geq K$ is equivalent to Lott-Villani and Sturm's weak curvature-dimension condition which extends the weighted Ricci curvature bounded from below on (weighted) Riemannian manifolds (\cite{Oh2}, \cite{Li}).
Studies show that there are essential differences between the case when $N\in [n, \infty)$ and the case when $N=\infty$. Some important progress has been made  on global analysis and topology on Finsler measure spaces under the assumption that ${\rm{Ric}}_N\geq K$ for $N\in [n, \infty)$( \cite{Sh1}-\cite{Sh2}, \cite{Oh1}-\cite{Oh2}, \cite{Xia1}-\cite{Xia4} and references therein) with the help of some nonlinear approaches.
However, there are few results under some conditions on Ric$_\infty$. In this paper, we shall prove that
 ${\rm{Ric}}_\infty=K$ is equivalent to that $(M, F, m)$ is a gradient Ricci soliton with soliton constant $K$ on a Finsler manifold (see \S \ref{sec3}).

\subsection{Conformal vector fields} \label{subsec22}

Let $V$ be a vector field on an $n$-dimensional differential manifold $M$ and $\{\varphi_t\}$  the local one-parameter transformation group on $M$ generated by $V$.
%$V$ is said to be {\it complete} if it generates a global one-parameter group of transformations on $M$. Such a vector field always exists on a compact manifold.
Denote by $\hat \varphi_t$ a lift of $\varphi_t$ on $TM$, i.e., $\hat \varphi_t(x,y):=(\varphi_t(x), (\varphi_t)_*(y))$. Then
$\{\hat\varphi_t\}$ forms a one-parameter transformation group on $TM$ and induces a vector field
\beq \hat V=V^i\frac{\partial}{\partial x^i}+y^j\frac{\partial V^i}{\partial x^j}\frac{\partial
}{\partial y^i} \label{hat-V}\eeq on $TM$. $\hat V$ is said to be a {\it complete lift} of $V$ on $TM$. If $\hat\varphi_t^*(F)=e^{2\sigma_t}F$ for some function $\sigma_t(x)$ on $M$ with $\sigma_0(x)=0$, then $V$ is called a {\it conformal vector field} of $F$ with conformal factor $c$, equivalently, $\hat V(F)=2cF$, where $c=c(x)=\frac {d\sigma_t}{dt}|_{t=0}$.  In particular, $V$ is called a {\it homothetic vector field} with dilation $c$ if $c$ is a constant and a {\it Killing vector field} if $c=0$ (\cite{SX}).

 For any vector field $V$ on $M$, the {\it Lie derivative} $\mathcal L$ of a geometric object $\Xi(x, y)$ on $TM$ with respect to the complete lift $\hat V$ of $V$ on $TM$ is defined by
 \beq  ({\mathcal L}_{\hat V}\Xi)_{(x,y)}=\lim\limits_{t\rightarrow 0}\frac{\hat\varphi_t^*(\Xi)-\Xi}{t}=\frac {d}{dt}\Big |_{t=0}\hat\varphi_t^*(\Xi).\label{L-der}\eeq
It is easy to see that
 ${\mathcal L}_{\hat V}(F)=\hat V(F)$.  Using $C_{ijk}(y)y^i=0$ and (\ref{g-com}), we obtain
 \beq {\mathcal L}_{\hat V}(F^2(y))&=&\hat V(F^2(y))= \frac{\partial}{\partial x^i}\left( g_{kl}y^ky^l\right)V^i+\frac{\partial}{\partial y^j}\left( g_{il}y^iy^l\right)\frac{\partial V^j}{\partial x^k}y^k\nonumber \\ &=&\frac{\partial g_{kl}}{\partial x^i}y^ky^lV^i+2g_{jl}\frac {\partial V^j}{\partial x^k}y^k y^l\nonumber \\
&=&2g_{jl}\Gamma^j_{ik}y^ky^lV^i+2g_{jl}\frac {\partial V^j}{\partial x^k}y^k y^l
=2V_{l|k}y^ly^k, \label{LV-F2} \eeq where  $V_i=g_{ij}(y)V^j$.
  \begin{prop}\label{prop21} (\cite{SX}) Let $V$ be a vector field on $(M, F)$. Then the following statements are equivalent.

(1) $V$ is a conformal vector field with the conformal factor $c$.

(2)${\mathcal L}_{\hat V}(F)=2cF$.

(3) In local coordinates,  $V=V^i\frac{\partial}{\partial x^i}$ satisfies
 \beqn V_{i|j}+V_{j|i}+2C_{ij}^pV_{p|q}y^q=4cg_{ij}, \label{CVF}\eeqn  where $C_{ij}^p=g^{pq}C_{ijq}$, $V_i=g_{ij}V^j$ and $``|"$ is the horizontal covariant derivative with respect to the Chern connection of $F$. \end{prop}

In particular, when $F=\alpha+\beta$ is a Randers metric, we have the following equivalent characterizations for conformal vector fields.
\begin{prop} \label{prop22}(\cite{SX})  Let $(M, F)$ be a Randers manifold with navigation data $(h, W)$. Then the following statements are equivalent.

(1)  $V$ is a conformal vector field of $F$  with conformal factor $c$.

(2) ${\mathcal L}_{\hat V}(\alpha)=2c\alpha, \ {\mathcal L}_{\hat V}(\beta)=2c\beta$.

(3) ${\mathcal L}_{\hat V}(h)=2ch, \ {\mathcal L}_{\hat V}(W_0)=2cW_0$.
\end{prop}

 A direct calculation gives
\beq {\mathcal L}_{\hat V}(\alpha^2)&=& 2V_{i;j}y^iy^j, \qquad
{\mathcal L}_{\hat V}(\beta)=(V^kb_{j;k}+b^kV_{k;j})y^j, \label{L-ab} \eeq where $V_i=a_{ij}V^j$ in this case.
Similarly,
\beq {\mathcal L}_{\hat V}(h^2)&= & 2V_{i:j}y^iy^j,\qquad
{\mathcal L}_{\hat V}(W_0)=(V^kW_{j:k}+W^kV_{k:j})y^j,\label{L-hW}\eeq where $V_i=h_{ij}V^j$ in this case. Thus (2) and (3) in Proposition \ref{prop22} are respectively rewritten as
\beq V_{i;j}+V_{j;i}=4ca_{ij}, \ \ \ V^jb_{i;j}+b^jV_{j;i}=2cb_i, \label{ab-i} \eeq
and
\beq \ \ \ \ \ V_{i:j}+V_{j:i}=4ch_{ij}, \ \ \  V^jW_{i:j}+W^jV_{j:i}=2cW_i.\label{hW-i}\eeq
Note that the meanings of $V_i$ in (\ref{LV-F2}), (\ref{ab-i}) and (\ref{hW-i}) are different. We abuse the notations without confusions.

\section{Almost Ricci solitons on Finsler manifolds} \label{sec3}

In this section, we shall prove Theorem \ref{thm11}.
Let us first recall he following result, which characterizes Finsler Ricci solitons as a fixed point to the unnormalized Ricci flow (\ref{UNRF}).
\begin{prop} (\cite{BY}) \label{prop31} Let $(M, F_0)$ be a compact Finsler manifold. Then there exists a function $\varrho=\varrho(t)$ of the time $t$ and a family of diffeomorphisms $\psi_t$ on $M$ such that $\hat F_t^2:=\varrho(t)\hat \psi_t^*(F_0^2)$ is a solution of (\ref{UNRF}) if and only if there is a vector field $V$ on $M$ such that $(F_0, V)$ satisfies (\ref{RS-kappa}) with soliton constant $\kappa$, where $\hat\psi_t$ is a global lift of $\psi_t$ on $TM$ and $\hat V$ is a complete lift of $V$. \end{prop}

 Let $Y$ be a non-vanishing $C^\infty$-geodesic field on an open subset $U\subset M$ and $dV_{g_Y}$ be the volume form of the Riemannian metric $g_Y$. Given a measure $m$ on $(M, F)$, the volume form  determined by the measure $m$ can be written as $dm=e^{-\psi}dV_{g_Y}$ on $U$, where
$\psi$ is given by
\beq \psi(x)=\log \frac{\sqrt{{\rm{det}}{(g_{ij}(x, Y_x))}}}{\sigma_F(x)}=\tau(x, Y_x),\label{psi} \eeq which is just the distortion along $Y_x$ at $x\in U$ (see (\ref{tau})). Obviously, $\psi\in C^\infty(U)$. We call $\psi$ the {\it weight function} of $m$ on $U$ with respect to $g_Y$. By definition,  the S-curvature of $(M, m)$ is given by
\beq S(x, Y_x)=Y_x\left[\tau(\cdot, Y)\right]=d\psi_x(Y_x).\label{S-Y}\eeq Thus,
\beq \dot S(x, Y_x)=Y_x\left[S(\cdot, Y)\right]=Y_x[Y(\psi)]={\rm Hess}(\psi)(Y_x)=\widehat{\rm Hess}(\psi)(Y_x),\label{dot-S-Y}\eeq where $\widehat{\rm Hess}(\psi)$ is the Hessian of $\psi$ with respect to $\hat g=g_Y$.

Let $V$ be a smooth vector field on $U$ with the complete lift $\hat V$ on $TU$. In local coordinates, we write $V=V^i\frac{\partial}{\partial x^i}$ and $Y=Y^i\frac{\partial}{\partial x^i}$. Note that $F^2(Y)=\hat g^2(Y)$ and
\beqn \frac{\partial \hat g_{ij}}{\partial x^k}=\frac{\partial g_{ij}}{\partial x^k}(Y)+2C_{ijl}(Y)\frac{\partial Y^l}{\partial x^k}.\eeqn
In the same way as (\ref{LV-F2}) and using the above equation, we have
\beqn \mathcal L_{\hat V}(F^2(Y))=\mathcal L_{\hat V}(\hat g^2(Y))=2V_{j||k}Y^jY^k, \eeqn where $V_i=g_{ij}(Y)V^j$ and ``$||$" means the covariant derivative with respect to the Levi-Civita connection of $\hat g$. In particular, when $V$ is the gradient field of the function $\psi$ with respect to $\hat g$, i.e., $V=$grad$_{\hat g}\psi$, we have
\beqn \mathcal L_{\hat V}(F^2(x, Y_x))=\mathcal L_{\hat V}(F^2(Y))|_x=2\widehat{\rm Hess}(\psi)(Y_x). \label{LH}\eeqn  From this and (\ref{dot-S-Y}), we have proved that
\begin{lem} \label{lem31} Let $(M, F, m)$ be an $n$-dimensional Finsler measure space and $Y$ be a geodesic field on an open subset $U\subset M$. Assume that $V$ is the gradient for the weight function $\psi$ of $m$ on $U$ with respect to $g_Y$, i.e., $V={\rm grad}_{g_Y}\psi$. Then $\mathcal L_{\hat V}(F^2(x, Y_x))=2\dot S(x, Y_x)$ for any $x\in U$.
\end{lem}

Next we prove Theorem \ref{thm11}, which actually gives a geometric explanation of gradient almost Ricci solitons.

\begin{proof}[Proof of Theorem~{\upshape\ref{thm11}}] For any $y\in T_xM\setminus\{0\}$, there is a neighborhood $U_x$ of $x$ in which $Y$ is a geodesic extension of $y$. Write $dm=e^{-\psi}dV_{g_Y}$ for some function $\psi$ on $U_x$ as in the introduction.

 Assume that Ric$_\infty=\kappa$, i.e., Ric+$\dot S=\kappa F^2$. Let $$\widetilde V:={\rm grad}_{g_{Y}}\psi,$$ which is a smooth vector field on $U$. Choose smaller neighborhoods $W_1$ and $W_2$ of $x$ such that the closure $\overline W_1$ is compact and $W_1\subset \overline W_1\subset W_2\subset \overline W_2\subset U$, and a function $f\in C^\infty(M)$ such that $f\equiv 1$ on $\overline W_1$ and $f\equiv 0$ outside $W_2$. We define a vector field $V$ on $M$
\beqn V(p)=\left\{\begin{array}{ll} f(p)\widetilde V(p), & p\in U \\ 0, & p\notin U. \end{array}\right.\eeqn Then $V$ is a smooth vector field  on $M$ with $V|_{W_1}=\widetilde V$. Consequently, $V_x=\widetilde V_x=({\rm grad}_{g_Y}\psi)(x)$. On the other hand, by the assumption, we have  \beqn {\rm Ric}(p, Y_p)+\dot S(p, Y_p)=\kappa
 F^2(p, Y_p)\label{RSF}\eeqn on $W_1$. From this, $ V|_{W_1}=\widetilde V$ and Lemma \ref{lem31}, we get $2{\rm Ric}(p, Y_p)+\mathcal L_{\hat{V}}(F^2(p, Y_p))=2\kappa F^2(p, Y_p)$ for any $p\in W_1$.  In particular, it holds
  \beq 2{\rm Ric}(x, Y_x)+\mathcal L_{\hat{V}}(F^2(x, Y_x))=2\kappa F^2(x, Y_x).\label{RSF-x}\eeq Note that $y=Y_x$. Then $(F, V)$ satisfies (\ref{RS-kappa}) at $(x, y)$. Hence $(M, F, m)$ is a gradient almost Ricci soliton.

Conversely, if $(M, F, m)$ is a gradient almost Ricci soliton, then there is a vector field $V$ on $M$ such that $V_x=({\rm grad}_{g_Y}\psi)(x)$  and $(F, V)$ satisfies (\ref{RS-kappa}) at $(x, y)\in TM$, which means that (\ref{RSF-x}) holds. Choose $W_1$, $W_2$ and $f$ as above and define
% smaller neighborhoods $W_1$ and $W_2$ of $x$ such that $\overline W_1$ is compact and $W_1\subset \overline W_1\subset W_2\subset \overline W_2\subset U_x$,  and a function $f\in C^\infty(M)$ such that $f\equiv 1$ on $\overline W_1$ and $f\equiv 0$ outside $W_2$ as above. Define
$$\widetilde V={\rm grad}_{g_{Y}}(f\psi),$$ namely, $\tilde V$ is the gradient of the weight function $f\psi=\psi$ on $W_1$ with respect to $g_Y$. Consequently,
$\mathcal L_{\hat{\tilde V}}(F^2(x, Y_x))=2\dot S(x, Y_x)$ by Lemma \ref{lem31}. Plugging this into (\ref{RSF-x}) with $\hat{\tilde V}$ instead of $\hat V$ and $Y_x=y$ yield Ric$_\infty(x, y)=\kappa F^2(x, y)$. By the arbitrariness of $(x, y)\in TM$, we have Ric$_\infty=\kappa$.  \end{proof}

\section{Almost Ricci solitons on Randers manifolds} \label{sec4}

Let $F=\alpha+\beta$ be a Randers metric, where $\alpha=\sqrt{a_{ij}(x)y^iy^j}$ and $\beta=b_i(x)y^i$ with $b:=\|\beta\|_\alpha<1$ are respectively the Riemannian metric and 1-form on $M$. It is known that the volume form determined by the Busemann-Hausdorff measure $m_{BH}$ on $M$ is given by $dm_{BH}=\sigma_{BH}dx$ (\S 5.2, \cite{ChS}), where
$$\sigma_{BH}=e^{(n+1)\log\sqrt{1-b^2}}\sigma_\alpha, \ \ \ \sigma_\alpha=\sqrt{\det(a_{ij})}.$$
\begin{lem} \label{lem-41}(\cite{ChS}, \cite{CS2}) Let $(M ,F, m_{BH})$ be an $n$-dimensional Randers manifold equipped with the Busemann-Hausdorff measure $m_{BH}$. For any scalar function $\sigma$ on $M$, the following statements are equivalent.

(1) $F$ is of isotropic S$_{BH}$-curvature $\sigma$, i.e., $S_{BH}=(n+1)\sigma F$.

(2) $e_{00}=2\sigma(\alpha^2-\beta^2)$.
\end{lem}

Next we shall give an equivalent characterization of almost Ricci solitons for Randers metrics. For this, we need some lemmas. The following lemma is obvious.
\begin{lem}\label{lem41} Let $F=\alpha+\beta$ be a Randers metric on an $n$-dimensional manifold $M$ and $V$ be a vector field on $M$ with complete lift $\hat V$ on $TM$. Then
\beq {\mathcal L}_{\hat V}(F^2)=\frac{F}{\alpha}{\mathcal L}_{\hat V}(\alpha^2)+2F{\mathcal L}_{\hat V}(\beta).\label{LV-F2*}\eeq
\end{lem}

\begin{lem} \label{lem42}(\cite{BR}, \cite{CS2}) Let $F=\alpha+\beta$ be a Randers metric on an $n$-dimensional manifold $M$. Then Ricci curvature of $F$ is given by
\beq {\rm{Ric}}={}^{\alpha}{\rm Ric}+2\alpha s^i_{\ 0;i}-2t_{00}-\alpha^2 t^i_{\ i}+(n-1)\Xi,\label{Ric-F} \eeq where ${}^{\alpha}{\rm{Ric}}$ denotes the Ricci curvature of $\alpha$ and  $\Xi$ is defined by
\beq \Xi:=\frac {2\alpha}F\left(q_{00}-\alpha t_0\right)+\frac 3{4F^2}(r_{00}-2\alpha s_0)^2-\frac 1{2F}(r_{00;0}-2\alpha s_{0;0}),\label{Xi}\eeq where $q_{00}=q_{ij}y^iy^j$, here $q_{ij}:= r_{ik}s^k_{\ j}$.
\end{lem}

Note that $(b^2)_{;j}=2(r_j+s_j)$, $s_ib^i=s^ib_i=0$ and $b^i_{\ ;k}=r^i_{\ k}+s^i_{\ k}$. It is straightforward to check the following equalities.
\begin{lem} \label{lem43} Let $F=\alpha+\beta$ be a Randers metric on an $n$-dimensional manifold $M$. Suppose that $e_{00}=2\sigma(x)(\alpha^2-\beta^2)$. Then
\beqn & & r_{ij}= - s_ib_j-s_jb_i+2\sigma(a_{ij}-b_ib_j), \ \ r^i_{\ j}=-s^ib_j-b^is_j+2\sigma(\delta^i_j-b^ib_j), \ \ r^j_{\ j}=2\sigma(n-b^2),\nonumber \\
& & r_j=-b^2s_j+2\sigma(1-b^2)b_j, \ \  r=2\sigma b^2(1-b^2), \ \  r^i_{\ 0}=-\beta s^i-b^is_0+2\sigma(y^i-\beta b^i), \nonumber \\
& & r_{00}=-2\beta s_0+2\sigma(\alpha^2-\beta^2), \ \ r^i_{\ i;0}=2\sigma_0(n-b^2)-4\sigma(1-b^2)(2\sigma\beta+s_0), \nonumber \\ & & r^i_{\ ;i}=-2(1-b^2)\left(s_is^i-\sigma_ib^i-2n\sigma^2+6\sigma^2b^2\right)-b^2s^i_{\ ;i}, \ \ q_{00}=-(s_0^2+t_0\beta+2\sigma\beta s_0),\nonumber \\ & & r_{00;0}=-2s_{0;0}\beta +4s_0^2\beta +8\sigma s_0\beta^2+2(\sigma_0-2\sigma s_0-4\sigma^2\beta)(\alpha^2-\beta^2),  \eeqn where $\sigma_0:=\sigma_{i}y^i$, $\sigma_i: =\sigma_{x^i}$.
\end{lem}

Based on these lemmas, we prove the following result.
\begin{thm} \label{thm41} Let $F=\alpha+\beta$ be a non-Riemannian Randers metric on an $n (\geq 2)$-dimensional manifold $M$ and $V$ be a vector field on $M$. Then $(M, F, V)$ is an almost Ricci soliton with soliton scalar $\kappa$ if and only if there are scalar functions $c$, $\sigma$ on $M$ such that ${\mathcal L}_{\hat V}(\alpha)=2c\alpha$ (i.e., $V$ is a conformal vector field of $\alpha$ with conformal factor $c$) and
\beq e_{00}&=& 2\sigma(\alpha^2-\beta^2), \label{e00}\\
 {}^{\alpha}{\rm{Ric}}&=& \kappa(\alpha^2+\beta^2)+(t^i_{\ i}-2c)\alpha^2+2t_{00}-(n-1)\sigma^2(3\alpha^2-\beta^2)
 \nonumber \\
& & -\left((n-1)\sigma_0+ {\mathcal L}_{\hat V}(\beta)\right)\beta-(n-1)(s_0^2+s_{0;0}), \label{a-Ric}\\
s^i_{\ 0;i}&=&\left(\kappa -c\right)\beta+(n-1)\left(\frac 12\sigma_0+t_0+2\sigma s_0+\sigma^2\beta\right)-\frac 12 {\mathcal L}_{\hat V}(\beta),\label{si0i}
\eeq
where $\hat V$ is the complete lift of $V$ and ${\mathcal L}_{\hat V}(\beta)$ is given by (\ref{L-ab})$_2$.
\end{thm}
\begin{proof}  Assume that $(M, F, V)$ is an almost Ricci soliton with soliton scalar $\kappa$. Then (\ref{RS-kappa}) holds. From this and Lemmas \ref{lem41}-\ref{lem42}, we have
\beq {{}^\alpha}{\rm{Ric}}+2\alpha s^i_{\ 0;i}-2t_{00}-\alpha^2t^i_{\ i}+(n-1)\Xi+\frac{F}{2\alpha}{\mathcal L}_{\hat V}(\alpha^2)+F{\mathcal L}_{\hat V}(\beta)=\kappa F^2,\label{RS1}\eeq where ${\mathcal L}_{\hat V}(\alpha^2)$ and ${\mathcal L}_{\hat V}(\beta)$ are given by (\ref{L-ab}).
By (\ref{Xi}), we have
\beq 4F^2\Xi &=&8\alpha F(q_{00}-\alpha t_0)+3(r_{00}-2s_0\alpha)^2-2F(r_{00;0}-2\alpha s_{0;0})\nonumber \\
&=& 3r_{00}^2-2\beta r_{00;0}-2\left(6r_{00}s_0-4\beta q_{00}+r_{00;0}-2\beta s_{0;0}\right)\alpha \nonumber \\
 & & +4\left(3s_0^2+2q_{00}-2\beta t_0+s_{0;0}\right)\alpha^2-8t_0\alpha^3.\label{F-Xi}\eeq
Multiplying both sides of (\ref{RS1}) by $4\alpha F^2$ and using (\ref{F-Xi}) yield
 \beq A_0+A_1\alpha+A_2\alpha^2+A_3\alpha^3+A_4\alpha^4+A_5\alpha^5=0, \label{A1-5-a*}\eeq
 equivalently,
\beq A_0+A_2\alpha^2+A_4\alpha^4+\alpha\left(A_1+A_3\alpha^2+A_5\alpha^4\right)=0,\label{A1-5-a}\eeq where $A_i(0\leq i\leq 5)$ are given by
\beqn A_0:&=& 2\beta^3{\mathcal L}_{\hat V}(\alpha^2),
\nonumber \\
A_1:&=&4\beta^2({}^{\alpha}{\rm{Ric}})-8\beta^2t_{00}+(n-1)(3r_{00}^2-2\beta r_{00;0})+6\beta^2{\mathcal L}_{\hat V}(\alpha^2)+4\beta^3{\mathcal L}_{\hat V}(\beta)-4\kappa \beta^4,
\nonumber \\
A_2:&=& 8\beta({}^{\alpha}{\rm{Ric}})+8\beta^2 s^i_{\ 0;i}-16\beta t_{00}-2(n-1)\left(6r_{00}s_0-4\beta q_{00}+r_{00;0}-2\beta s_{0;0}\right)\nonumber \\
 & &+6\beta{\mathcal L}_{\hat V}(\alpha^2) +12\beta^2{\mathcal L}_{\hat V}(\beta)-16\kappa\beta^3,
\nonumber \\
A_3:&=& 4({}^{\alpha}{\rm{Ric}})+16\beta s^i_{\ 0;i}-8t_{00}-4\beta^2 t^i_{\ i}+4(n-1)\left(3s_0^2+2q_{00}-2\beta t_0+s_{0;0}\right)\nonumber \\
& &+2{\mathcal L}_{\hat V}(\alpha^2)+12\beta {\mathcal L}_{\hat V}(\beta)-24\kappa\beta^2,
\nonumber \\
A_4:&=& 8s^i_{\ 0;i}-8\beta t^i_{\ i}-8(n-1)t_0+4{\mathcal L}_{\hat V}(\beta)-16\kappa \beta,
\nonumber \\
A_5:&=& -4t^i_{\ i}-4\kappa.
\eeqn
Since $\alpha^2=a_{ij}(x)y^iy^j$ is positive definite, $\alpha$ is an irrational function in $y$. Moreover,  $A_i (0\leq i\leq 5)$ are polynomials in $y$. Consequently, (\ref{A1-5-a}) are equivalent to the following two equations:
\beq & A_0+A_2\alpha^2+A_4\alpha^4=0, \label{A-024}\\
& A_1+A_3\alpha^2+A_5\alpha^4=0.\label{A-135}\eeq
(\ref{A-135})$\times \alpha^2$-(\ref{A-024})$\times \beta$ yields
\beq & & 2(\alpha^2-\beta^2)\Big\{2\alpha^2({}^{\alpha}{\rm{Ric}})-4\alpha^2t_{00}+4\alpha^2\beta s^i_{\ 0;i}-2\alpha^4t^i_{\ i}
+4(n-1)\alpha^2q_{00}+2(n-1)\alpha^2s_{0;0}\nonumber \\ & & +(\alpha^2+\beta^2){\mathcal L}_{\hat V}(\alpha^2)+4\alpha^2\beta{\mathcal L}_{\hat V}(\beta)-2\kappa\alpha^2(\alpha^2+3\beta^2)+6(n-1)\alpha^2s_0^2\Big\}\nonumber \\ & & +3(n-1)\alpha^2(r_{00}+2s_0\beta)^2=0. \label{RS2}\eeq
 In the same way as above, $\alpha^2-\beta^2$ is positive on $TM_0$ because $b=\|\beta\|_\alpha<1$. We deduce that $\alpha^2-\beta^2$ is an irreducible polynomial in $y$. Therefore, from (\ref{RS2}) and $n\geq 2$, there exists a function $\sigma$ on $M$ (i.e., independent of $y$) such that
\beq r_{00}+2s_0\beta=2\sigma(\alpha^2-\beta^2),  \label{r00-s0}\eeq that is, (\ref{e00}) holds.   On the other hand, since the second and third terms on LHS of (\ref{A-024}) include $\alpha^2$ respectively,  $A_0$ must be divided by $\alpha^2$. Note that $A_0= 2\beta^3{\mathcal L}_{\hat V}(\alpha^2)$ and $\beta\neq 0$ since $F$ is non-Riemannian.  Hence there is a scalar function $c$ on $M$ such that
\beq {\mathcal L}_{\hat V}(\alpha^2)=4c\alpha^2, \ \ \ {\rm equivalently}, \ \ V_{i;j}+V_{j;i}=4ca_{ij}.\label{V-conf}\eeq
This means that $V$ is a conformal vector field of $\alpha$ with conformal factor $c$.
Plugging (\ref{r00-s0})-(\ref{V-conf}) into (\ref{RS2}) gives
\beq {{}^\alpha}{\rm{Ric}}&=&\kappa(\alpha^2+3\beta^2)+2t_{00}-2\beta s^i_{\ 0;i}+\alpha^2t^i_{\ i}-2(n-1)q_{00}-(n-1)s_{0;0}\nonumber \\
& & -3(n-1)s_0^2-3(n-1)\sigma^2(\alpha^2-\beta^2)-2c(\alpha^2+\beta^2)-2\beta{\mathcal L}_{\hat V}(\beta). \label{a-Ric-n}\eeq
From this and Lemma \ref{lem43}, we get
\beqn A_2&=&8\kappa\beta(\alpha^2+\beta^2)+4(\alpha^2-\beta^2)\Big\{2c\beta-2(n-1)\sigma^2\beta-4(n-1)\sigma s_0-(n-1)\sigma_0\Big\} \nonumber \\ & &-8\beta^2s^i_{\ 0;i}+8\alpha^2\beta t^i_{\ i}+8(n-1)t_0\beta^2-8c\beta^3-4\beta^2\mathcal L_{\hat V}(\beta). \label{A2}\eeqn
Moreover, $A_0=8c\alpha^2\beta^3$ by (\ref{V-conf}). From these and (\ref{A-024}), one obtains
\beqn 4(\alpha^2-\beta^2)\alpha^2\Big\{-2(\kappa-c)\beta-(n-1)\left(2\sigma^2\beta+4\sigma s_0+\sigma_0+2 t_0\right)+2s^i_{\ 0;i}+\mathcal L_{\hat V}(\beta)\Big\}=0,\eeqn which implies (\ref{si0i}).
 Putting (\ref{si0i}) back in (\ref{a-Ric-n}) and using $q_{00}=-s_0^2-t_0\beta-2\sigma \beta s_0$ yield (\ref{a-Ric}).

The proof of sufficiency is straightforward. Assume that there are functions $c$ and $\sigma$ on $M$ such that ${\mathcal L}_{\hat V}(\alpha^2)=4c\alpha^2$ and (\ref{e00})-(\ref{a-Ric}) hold.  From (\ref{Xi})-(\ref{e00}) and Lemma \ref{lem43}, one obtains
\beq \Xi&=& -\frac{2\alpha}F\left(s_0^2+2\sigma s_0\beta+Ft_0\right)+3\left(s_0-\sigma(\alpha-\beta)\right)^2\nonumber \\
& & -\frac 1{F}\left\{-Fs_{0;0}+2s_0^2\beta+4\sigma s_0\beta^2
+F(\sigma_0-2\sigma s_0-4\sigma^2\beta)(\alpha-\beta)\right\}\nonumber \\
&=&-2(t_0+2\sigma s_0)\alpha+s_0^2+s_{0;0}+\sigma^2(\alpha-\beta)(3\alpha+\beta)-\sigma_0(\alpha-\beta).\label{Xi-formula}\eeq
Plugging this and (\ref{a-Ric})-(\ref{si0i}) in (\ref{Ric-F}) yields
 \beq {\rm Ric}=\kappa F^2-2c\alpha F-F{\mathcal L}_{\hat V}(\beta).\label{Ric-Fab}\eeq On the other hand, by Lemma \ref{lem41} and ${\mathcal L}_{\hat V}(\alpha^2)=4c\alpha^2$, we have ${\mathcal L}_{\hat V}(F^2)=4c\alpha F+2F{\mathcal L}_{\hat V}(\beta)$. From this and (\ref{Ric-Fab}), we get (\ref{RS-kappa}), i.e., $(M, F, V)$ is an almost Ricci soliton with soliton scalar $\kappa$.\end{proof}

To prove Theorem \ref{thm12}, we need the following lemmas.
\begin{lem}\label{lem45} Assume that $e_{00}=2\sigma (\alpha^2-\beta^2)$ for some function $\sigma$ on $M$. Then
\beq & & b^js_{j;0}=-t_0-2\sigma s_0+(s_js^j)\beta, \label{bs-1}\\
 & & b^js_{0;j}=-t_0+2\sigma s_0+2\sigma_0b^2-2(\sigma_jb^j)\beta+(s_js^j)\beta, \label{bs-2} \\
 & & r^i_{\ 0;i}= -s^i_{\ ;i}\beta+2t_0-2\sigma(n+1-2b^2)(s_0+2\sigma \beta)+2\sigma_0(1-b^2), \label{ri0i} \\
& & r^i_{\ ;i}-r^i_{\ j}r^j_{\ i}-b^jr^i_{\ i;j}=-2s_is^i-b^2s^i_{\ ;i}-2(n-1)(\sigma_jb^j)-4(n-1)\sigma^2b^2. \label{rrrr} \eeq
\end{lem}
\begin{proof} Since $s_jb^j=b_js^j=0$, by Lemma \ref{lem43}, we have
\beqn b^js_{j;0}=-s^jb_{j;0}=-s^j(r_{j0}+s_{j0})=-t_0-2\sigma s_0+(s_js^j)\beta.\eeqn Similarly, by the same calculation as (7.27) in \cite{CS2}, we have (\ref{bs-2}).  Note that $\beta_{;j}=r_{0j}+s_{0j}$, $b^i_{\ ;i}=r^i_{\ i}$ and $(b^2)_{;j}=2(r_j+s_j)$.  By Lemma \ref{lem43} and (\ref{bs-2}), we have
 \beqn r^i_{\ 0;i}&=&-s^i_{\ ;i}\beta-r_{0i}s^i+t_0-(s_0+2\sigma\beta)r^i_{\ i}-b^is_{0;i}+2\sigma_0-2(\sigma_ib^i)\beta-2\sigma(r_0-s_0)\nonumber \\
 &=& -s^i_{\ ;i}\beta+2t_0-2\sigma(n+1-2b^2)(s_0+2\sigma \beta)+2\sigma_0(1-b^2). \eeqn
Moreover, by Lemma \ref{lem43} again, we get
\beqn r^i_{\ j}r^j_{\ i}=2b^2(s_is^i)+4\sigma^2(n-2b^2+b^4), \ \ \ \ \   b^jr^i_{\ i;j}=2(n-b^2)(\sigma_jb^j)-8\sigma^2b^2(1-b^2).\eeqn From these and the formula for $r^i_{\ ;i}$ in Lemma \ref{lem43}, one obtains (\ref{rrrr}).
\end{proof}

\begin{lem} \label{lem44} Under (\ref{e00}) and (\ref{a-Ric}), we have
\beq s^i_{\ 0;i}&=&(\kappa-2c)\beta+(n-1)\left\{\sigma^2\beta+t_0+2\sigma s_0-\frac{3b^2-4}2\sigma_0-\frac{3(1-b^2)}{2(1+b^2)}(\sigma_ib^i)\beta\right\}\nonumber \\ & &+\frac{2cb^2}{1+b^2}\beta-\frac {b^2}2{\mathcal L}_{\hat V}(\beta)-\frac{1-b^2}{2(1+b^2)}\left(V^kb_{j;k}+b^kV_{k;j}\right)b^j\beta.\label{si0i*}\eeq
\end{lem}
\begin{proof} Let ${{}^\alpha}{\rm Ric}_{ij}$ be the Ricci tensor of $\alpha$. By (7.21) in \cite{CS2}, we have
\beq  s^i_{\ 0;i}={{}^\alpha}{\rm{Ric}}_{0j}b^j+r^i_{\ i;0}-r^i_{\ 0;i}, \label{si0i**}\eeq where  ${{}^\alpha}{\rm{Ric}}_{0j}={{}^\alpha}{\rm{Ric}}_{ij}y^i$.  From this, we have $ s^i_{\ j;i}={{}^\alpha}{\rm{Ric}}_{jl}b^l+r^i_{\ i;j}-r^i_{\ j;i}$. Thus
\beq s^i_{\ ;i}&=&-(b^js^i_{\ j})_{;i}=-(r^j_{\ i}+s^j_{\ i})s^i_{\ j}-b^js^i_{\ j;i}\nonumber \\
&=&-t_{\ j}^j-b^jb^l({{}^\alpha}{\rm{Ric}}_{jl})+r^i_{\ ;i}-r^i_{\ j}r^j_{\ i}-b^jr^i_{\ i;j}. \label{sii}\eeq
 By (\ref{a-Ric}), we have
 \beq {{}^\alpha}{\rm{Ric}}_{0j}b^j&=&\frac 12\left[{{}^\alpha}{\rm{Ric}}\right]_{y^iy^j}y^i b^j
=\frac 12\left[{{}^\alpha}{\rm{Ric}}\right]_{y^j}b^j \nonumber \\
 &=&\kappa(1+b^2)\beta+(t^i_{\ i}-2c)\beta+2t_0-(n-1)\sigma^2(3-b^2)\beta-\frac {1}2 b^2{\mathcal L}_{\hat V}(\beta)\label{Ric-bj} \\
 & &-\frac 12(n-1)\left\{\sigma_0b^2+(\sigma_jb^j)\beta+ b^js_{j;0}+b^js_{0;j}\right\} -\frac 12\left(V^kb_{j:k}+b^kV_{k;j}\right)b^j\beta,\nonumber \eeq
where we used $s_jb^j=b_js^j=0$.
Plugging (\ref{bs-1})-(\ref{bs-2}) in (\ref{Ric-bj}) leads to
 \beq {{}^\alpha}{\rm{Ric}}_{0j}b^j&=&\kappa(1+b^2)\beta+(t^i_{\ i}-2c)\beta-(n-1)\left\{\sigma^2(3-b^2)+s_js^j-\frac 12(\sigma_jb^j)\right\}\beta\nonumber \\
 & &-\frac 32(n-1)\sigma_0b^2+(n+1)t_0 -\frac 12b^2{\mathcal L}_{\hat V}(\beta)-\frac 12\left(V^kb_{j:k}+b^kV_{k;j}\right)b^j\beta.\label{Ric-bj-1}\eeq
Differentiating  this equation with respect to $y^l$ and then contracting this with $b^l$ yield
\beq b^jb^l({{}^\alpha}{\rm Ric}_{jl})&=&\kappa(1+b^2)b^2+b^2(t^i_{\ i}-2c)-(n-1)\left\{\sigma^2(3-b^2)+s_js^j\right\}b^2 \nonumber \\
& &-(n-1)(\sigma_lb^l)b^2-(n+1)s_ks^k- b^2\left(V^kb_{l:k}+b^kV_{k;l}\right)b^l.\label{Ric-bij} \eeq
Since $r^j_{\ i}s^i_{\ j}=a^{jk}r_{ki}a^{il}s_{lj}=a^{jk}r^l_{\ k}s_{lj}=-r^l_{\ k}s^k_{\ l}$, we have $r^j_{\ i}s^i_{\ j}=0$.
 Inserting (\ref{rrrr}) and (\ref{Ric-bij}) in (\ref{sii}) gives
\beqn s^i_{\ ;i}&=&-t^i_{\ i}-\kappa b^2+(n-1)\left(s_is^i-\sigma^2b^2+\frac{b^2-2}{1+b^2}(\sigma_ib^i)\right)\\
& &+\frac{2cb^2}{1+b^2}+\frac{b^2}{1+b^2} \left(V^kb_{l;k}+b^kV_{k;l}\right)b^l. \eeqn From this and (\ref{ri0i}), one obtains
\beqn r^i_{\ 0;i}&=& (\kappa b^2+t^i_{\ i})\beta+2t_0-2\sigma(n+1-2b^2)( s_0+2\sigma\beta)+2\sigma_0(1-b^2)-\frac{2cb^2}{1+b^2}\beta \nonumber \\
 & &+(n-1)\left(\sigma^2b^2-s_is^i+\frac{2-b^2}{1+b^2}(\sigma_ib^i)\right)\beta -\frac{b^2}{1+b^2}\left(V^kb_{j;k}+b^kV_{k;j}\right)b^j\beta. \eeqn
 Since $r^i_{\ i;0}=2\sigma_0(n-b^2)-4\sigma(1-b^2)(s_0+2\sigma\beta)$ by Lemma \ref{lem43}, we have
 \beqn r^i_{\ i;0}-r^i_{\ 0;i}&=& -2t_0-(\kappa b^2+t^i_{\ i})\beta +(n-1)\left(4\sigma^2-\sigma^2b^2+s_is^i-\frac{2-b^2}{1+b^2}(\sigma_ib^i)\right)\beta\nonumber \\
 & & +2(n-1)\left(\sigma_0+\sigma s_0\right)+\frac{2cb^2}{1+b^2}\beta +\frac{b^2}{1+b^2}\left(V^kb_{j;k}+b^kV_{k;j}\right)b^j\beta. \label{rr}\eeqn
Plugging this and (\ref{Ric-bj-1}) into (\ref{si0i**}) yields (\ref{si0i*}).
\end{proof}

\begin{lem}\label{lem47} Under (\ref{e00}) and (\ref{a-Ric}), the equations (\ref{si0i}) and (\ref{tau-0-1}) are equivalent. \end{lem}
\begin{proof}  Assume that (\ref{si0i}) holds. By Lemma \ref{lem44}, the right hand sides of (\ref{si0i}) and (\ref{si0i*}) are equal, which means that
\beq 3(n-1)\sigma_0=\frac{2c}{1+b^2}\beta+\frac{3(n-1)}{1+b^2}(\sigma_ib^i)\beta-{\mathcal L}_{\hat V}(\beta)+\frac{1}{1+b^2}\left(V^kb_{j;k}+b^kV_{k;j}\right)b^j\beta.\label{tau-0}\eeq
Differentiating this with respect to $y^i$ and then contracting this with $b^i$ yield
\beq 3(n-1)\sigma_ib^i=2cb^2-\left(V^kb_{j;k}+b^kV_{k;j}\right)b^j.\label{tau-b}\eeq
Inserting (\ref{tau-b}) into (\ref{tau-0}) gives (\ref{tau-0-1}).

Conversely, assume that (\ref{tau-0-1}) holds. Then
\beqn 3(n-1)\sigma_i=2cb_i-\left(V^kb_{i;k}+b^kV_{k;i}\right),\eeqn which implies
(\ref{tau-b}). Plugging (\ref{tau-b}) and (\ref{tau-0-1}) into (\ref{si0i*}) leads to (\ref{si0i}).
\end{proof}

 \begin{proof}[Proof of Theorem~{\upshape\ref{thm12}}] It follows from Theorem \ref{thm41}, Lemma \ref{lem47} and substituting $\mathcal L_{\hat V}(\beta)$ in (\ref{a-Ric}) with $\mathcal L_{\hat V}(\beta)=2c\beta-3(n-1)\sigma_0$. \end{proof}

\section{Navigation description of Randers almost Ricci solitons} \label{sec5}

In this section, we shall give a navigation description of almost Ricci solitons for Randers metrics $F=\alpha+\beta$. Given a navigation data $(h, W)$ of $F$, we express $F$ in the form (\ref{F-hW}),  where
\beq a_{ij}=\frac {h_{ij}}{\lambda}+\frac{W_iW_j}{\lambda^2}, \ \ \ \  b_i=-\frac {W_i}{\lambda}, \label{ab}\eeq where $W_i=h_{ij}W^j$ and $\lambda =1-\|W\|_h^2$. Then the components of inverse matrix $(a_{ij})^{-1}$ and $b^i=a^{ij}b_j$ are respectively given by
\beq a^{ij}=\lambda(h^{ij}-W^iW^j), \ \ \ \ b^i=-\lambda W^i. \label{ab-hw}\eeq Conversely, given a Riemannian metric $\alpha=\sqrt{a_{ij}(x)y^iy^j}$ and a 1-form $\beta=b_i(x)y^i$, we also have
\beqn h_{ij}=\lambda(a_{ij}-b_ib_j), \ \ \ \ W^i=-\frac {b^i}{\lambda}. \eeqn We refer to \cite{BRS} for more details.
 \begin{lem} \label{lem51} (\cite{Xing}) Let $(M, F)$ be a non-Riemannian Randers metric with navigation data $(h, W)$. Then $F$ is of isotropic S$_{BH}$-curvature $\sigma$ if and only if $W$ is a conformal vector field of $h$ with conformal factor $-\sigma$, i.e.,  $W$ satisfies $ W_{j:k}+W_{k:j}=-4\sigma h_{jk}.$ \end{lem}

 Let $\xi:=y-F(x, y)W$ and $\tilde h:=\sqrt{h(\xi, \xi)}$. Observe that
   \beq  h^2-2FW_0=\lambda F^2 \label{hWF}\eeq from (\ref{F-hW}). Thus
\beq \tilde h^2=h(\xi, \xi)=h(y-FW, y-FW)=h^2-2FW_0+F^2\|W\|_h^2=F^2,\label{tilde-hF}\eeq namely, $\tilde h(x, \xi)=F(x, y)$. For the sake of convenience, we use `` ${\widetilde{(\cdot)}_0}$ " to denote the contraction with $\xi$, for example, $\widetilde W_0:=W_i\xi^i$, $\widetilde V_0: =V_i\xi^i$, $\widetilde V_{0:0}:=V_{j:k}\xi^j\xi^k$, $\tilde f_0=f_i\xi^i$, $\widetilde f_{;00}=f_{:ij}\xi^i\xi^j$  etc..
\begin{lem} \label{lem52} Let $V=(V^i)$ be a smooth vector field on $M$ and $\hat V$ be its complete lift on $TM$. Then
\beq {\mathcal L}_{\hat V}(\tilde h^2)=\frac 2{\tilde h+\widetilde W_0}\left\{\tilde h \widetilde V_{0:0}+\tilde h^2 \left(V_{j:k}W^k-W_{j:k}V^k\right)\xi^j\right\}\label{LVh2}\eeq for any $\xi\neq 0$.
\end{lem}

\begin{proof} By a direct calculation, we have
\beq {\mathcal L}_{\hat V}(\tilde h^2)&=&\hat V(h_{ij}\xi^i\xi^j)=\hat V(h_{ij})\xi^i\xi^j+2h_{ij}\hat V(\xi^i)\xi^j \nonumber \\
&=& V^k\frac{\partial h_{ij}}{\partial x^k}\xi^i\xi^j+2h_{ij}\left(y^k\frac{\partial V^i}{\partial x^k}-{\mathcal L}_{\hat V}(F) W^i-FV^k\frac{\partial W^i}{\partial x^k}\right)\xi^j.\label{LV-h} \eeq Denote by ${}^h\Gamma_{ij}^l$ the Riemannian connection coefficients of $h$. Then
$$\frac{\partial h_{ij}}{\partial x^k}=h_{il}{}^h\Gamma_{jk}^l+h_{jl}{}^h\Gamma_{ik}^l.$$ Plugging this and $y^k=\xi^k+FW^k$ in (\ref{LV-h}) yields
\beqn {\mathcal L}_{\hat V}(\tilde h^2)&=& 2h_{ij}V^i_{\ :k}\xi^j\xi^k+2Fh_{ij}\frac{\partial V^i}{\partial x^k}W^k\xi^j-2\widetilde W_0{\mathcal L}_{\hat V}(F)-2Fh_{ij}\frac{\partial W^i}{\partial x^k}\xi^j V^k\nonumber \\
&=& 2\widetilde V_{0:0}+2F\left(V_{j:k}W^k-W_{j:k}V^k\right)\xi^j-2\widetilde W_0{\mathcal L}_{\hat V}(F).\label{LV-h-1} \eeqn
Note that $F(x, y)=\tilde h(x, \xi)$ and ${\mathcal L}_{\hat V}(F)=(2\tilde h)^{-1}{\mathcal L}_{\hat V}(\tilde h^2)$. Then we have
\beq \left(\tilde h+\widetilde W_0\right){\mathcal L}_{\hat V}(\tilde h^2)=2\tilde h \widetilde V_{0:0}+2\tilde h^2 \left(V_{j:k}W^k-W_{j:k}V^k\right)\xi^j.\label{LV-h-1} \eeq Note that $\tilde h+\widetilde W_0\neq 0$ for any $\xi\neq 0$. Otherwise, we have $\tilde h^2=\widetilde W_0^2$. Differentiating this equation twice in $\xi$ yields $h_{ij}=W_iW_j$, which is impossible since the matrix $(h_{ij})$ is positive definite. Thus (\ref{LVh2}) follows from (\ref{LV-h-1}).
%(\ref{LVh2*}) directly follows from (\ref{LVh2}) and Lemma \ref{lem51}.
\end{proof}

\begin{lem}\label{lem53}(\cite{CS1}) Let $F$ be a Randers metric expressed by (\ref{F-hW}). Suppose that $F$ has isotropic S$_{BH}$-curvature $\sigma$. Then,  for any scalar function $\tilde\mu$ on $M$,
\beq {\rm{Ric}}-(n-1)\left(\frac{3\sigma_0}{F}+\tilde\mu-\sigma^2-2\sigma_{i}W^i\right)F^2=\widetilde{\rm{Ric}}-(n-1)\tilde\mu \tilde h^2,\label{Ric-Ric}\eeq where $\widetilde{\rm{Ric}}={}^h{\rm Ric}_{ij}\xi^i\xi^j$.
\end{lem}

Note that $F(x, y)=\tilde h(x, \xi)$ and we use $\tilde\kappa$ instead of $(n-1)\tilde\mu$.  (\ref{Ric-Ric}) is equivalent to
\beq & & 2{\rm{Ric}}+\mathcal L_{\hat V}(F^2)-6(n-1)\sigma_0F-2\Big\{\tilde \kappa-(n-1)\sigma^2-2(n-1)\sigma_{i}W^i\Big\}F^2\nonumber \\
& & =2\widetilde{\rm{Ric}}+ \mathcal L_{\hat V}(\tilde h^2)-2\tilde\kappa \tilde h^2 \label{Ric-Ric-LL}\eeq for any scalar function $\tilde\kappa$ on $M$. Based on this and Lemmas \ref{lem51}-\ref{lem52}, we prove Theorem \ref{thm13*}.

 \begin{proof}[Proof of Theorem~{\upshape\ref{thm13*}}]  Assume that $(M, F, V)$ is an almost Ricci soliton with soliton scalar $\kappa$. Since $F$ is non-Riemannian and $n\geq 2$, $S_{BH}=(n+1)\sigma F$ for some function $\sigma$ on $M$ by Corollary \ref{cor11}. Hence $W$ is a conformal vector field of $h$ with conformal factor $-\sigma$ by Lemma \ref{lem51}.  Let $\tilde\kappa:=\kappa+(n-1)\sigma^2+2(n-1)\sigma_iW^i$. By (\ref{Ric-Ric-LL}) and (\ref{RS-kappa}),  we have
\beq-6(n-1)\sigma_0F=2\widetilde{\rm{Ric}}+ \mathcal L_{\hat V}(\tilde h^2)-2\tilde\kappa \tilde h^2.\label{Ric-Ric-LL*}\eeq For any $\xi\neq 0$, plugging (\ref{LVh2}) in (\ref{Ric-Ric-LL*}) and using $\sigma_0=\sigma_i\xi^i+\tilde h\sigma_iW^i$ yield
\beq -3(n-1)(\tilde h^2+\tilde h\widetilde W_0)(\sigma_i\xi^i+\tilde h\sigma_iW^i)&=&(\tilde h+\widetilde W_0) {}^h{\rm Ric}_{ij}\xi^i\xi^j+\tilde h\widetilde V_{0:0} \nonumber \\
&  &+\tilde h^2\left(V_{j:k}W^k-W_{j:k}V^k\right)\xi^j-\tilde\kappa\tilde h^2(\tilde h+\widetilde W_0). \label{hW-xi} \eeq This equation holds obviously if $\xi=0$. Consequently, (\ref{hW-xi}) holds for any $(x, \xi)\in TM$ and it can be rewritten as the following form:
$$P+Q\tilde h=0,$$ where $P, Q$ are polynomials in $\xi$. From this, we get $P=Q=0$ since $\tilde h=\sqrt{h_{ij}(x)\xi^i\xi^j}$ is an irrational function in $\xi$. Based on this, (\ref{hW-xi}) is equivalent to the following two equations:
\beq & -3(n-1)\left\{(\sigma_iW^i)\tilde h^2+\tilde\sigma_0\widetilde W_0\right\} ={}^h{\rm Ric}_{ij}\xi^i\xi^j+\widetilde V_{0:0}-\tilde\kappa\tilde h^2, \label{h-Ric-V}\\
& -3(n-1)\left\{\tilde\sigma_0+(\sigma_iW^i)\widetilde W_0\right\}\tilde h^2= \widetilde W_0 {}^h{\rm Ric}_{ij}\xi^i\xi^j+\tilde h^2\left(V_{j:k}W^k-W_{j:k}V^k\right)\xi^j-\tilde\kappa \widetilde W_0\tilde h^2, \label{h-Ric-V*} \eeq where $\tilde\sigma_0:=\sigma_i\xi^i$.
Since $F$ is non-Riemannian, i.e., $W\neq 0$, (\ref{h-Ric-V*}) implies that ${}^h{\rm Ric}_{ij}\xi^i\xi^j$ is divided by $\tilde h^2$. There is a function $\mu$ on $M$ such that ${}^h{\rm Ric}_{ij}\xi^i\xi^j=\mu \tilde h^2$, namely, ${}^h{\rm Ric}_{ij}=\mu h_{ij}$, which means that $h$ is Einstein.  Differentiating (\ref{h-Ric-V})-(\ref{h-Ric-V*}) in $\xi$ leads to
 \beq & &V_{i:j}+V_{j:i}=2(\tilde\kappa-\mu) h_{ij}-3(n-1)\left\{2(\sigma_kW^k)h_{ij}+\sigma_iW_j+\sigma_jW_i\right\},  \label{VV-h}\\
 & &V_{j:k}W^k-W_{j:k}V^k= (\tilde \kappa-\mu) W_j-3(n-1) \left(\sigma_j+(\sigma_kW^k)W_j\right). \label{VW-WV}\eeq From  (\ref{VV-h})-(\ref{VW-WV}), one obtains
 \beq W_{j:k}V^k+V_{k:j}W^k=\left(\tilde\kappa-\mu\right)W_j-3(n-1)\left\{2(\sigma_iW^i) W_j-\lambda\sigma_j\right\}.\label{VW-h}\eeq
It follows from (\ref{L-hW}) that (\ref{VV-h}) and (\ref{VW-h}) are equivalent to
\beqn \mathcal L_{\hat V}(h^2)&=&2(\tilde\kappa-\mu)h^2-6(n-1)\left\{(\sigma_iW^i)h^2+\sigma_0W_0\right\}, \\
 \mathcal L_{\hat V}(W_0)&=&(\tilde\kappa-\mu)W_0-3(n-1)\left\{2(\sigma_iW^i)W_0-\lambda\sigma_0\right\},\eeqn which imply (\ref{LV-h2})-(\ref{LVW0}) since $\tilde\kappa-\mu=\kappa-\mu+(n-1)\sigma^2+2(n-1)\sigma_iW^i=c$.

 Conversely, assume that  ${}^h{\rm Ric}_{ij}=\mu h_{ij}$, $W$ is a conformal vector field of $h$ with conformal factor $\sigma$ and (\ref{LV-h2})-(\ref{LVW0}) hold. Then $S_{BH}=(n+1)\sigma F$ and $\widetilde{\rm Ric}=\mu\tilde h^2$. Differentiating (\ref{LV-h2})-(\ref{LVW0}) twice in $y$ gives (\ref{VV-h}) and (\ref{VW-h}), where $\tilde\kappa: =c+\mu=\kappa+(n-1)\sigma^2+2(n-1)\sigma_iW^i$. From these, we have
 \beqn & \widetilde V_{0:0}=(\tilde\kappa -\mu)\tilde h^2-3(n-1)\left(\tilde \sigma_0\widetilde W_0+(\sigma_iW^i)\tilde h^2\right), \nonumber \\
 & \left(V_{j:k}W^k-W_{j:k}V^k\right)\xi^j=(\tilde\kappa -\mu)\widetilde W_0-3(n-1)\left(\tilde \sigma_0+(\sigma_iW^i)\widetilde W_0\right). \eeqn
Inserting these in (\ref{LVh2}) yields
 \beqn {\mathcal L}_{\hat V}(\tilde h^2)=2(\tilde \kappa-\mu)\tilde h^2-6(n-1)\left(\tilde h\tilde \sigma_0+(\sigma_iW^i)\tilde h^2\right).\eeqn Consequently,
 \beqn 2\widetilde{\rm{Ric}}+ \mathcal L_{\hat V}(\tilde h^2)&=&2\tilde\kappa \tilde h^2 - 6(n-1)F\sigma_i\left(\xi^i+FW^i\right)\\
 &=& 2\tilde\kappa \tilde h^2-6(n-1)\sigma_0F, \eeqn which implies that $2{\rm Ric}+\mathcal L_{\hat V}(F^2)=2\kappa F^2$ by (\ref{Ric-Ric-LL}), i.e., $(M, F, V)$ is an almost Ricci soliton with soliton scalar $\kappa$.  \end{proof}

 \begin{proof}[Proof of Corollary~{\upshape\ref{cor13*}}] The first claim follows directly from Theorem \ref{thm13*}. Now we assume that $\mathcal L_{\hat V}(h^2)=2ch^2$. Then $(\sigma_iW^i)h^2+\sigma_0W_0=0$ by (\ref{LV-h2}), which implies that $\sigma$ is constant since $h^2$ is an irreducible polynomial in $y$. From (\ref{LVW0}), we have $\mathcal L_{\hat V}(W_0)=c W_0$.

If $\mathcal L_{\hat V}(W_0)=c W_0$ with $\|W\|_h^2\neq \frac 13$,  then \beq 2(\sigma_iW^i)W_0-\lambda\sigma_0=0 \label{cor-sw} \eeq by (\ref{LVW0}). Differentiating this in $y^k$ and then contracting with $W^k$ yields $(3\|W\|_h^2-1)(\sigma_kW^k)=0$ and hence $\sigma_kW^k=0$. Inserting this back into (\ref{cor-sw}) gives $\sigma_0=0$, i.e, $\sigma$=constant, which means that $\mathcal L_{\hat V}(h^2)=2ch^2$ by (\ref{LV-h2}).

In any case, we have $\mathcal L_{\hat V}(h^2)=2ch^2$ and $\mathcal L_{\hat V}(W_0)=c W_0$. This means that $V$ is a conformal vector field of $F$ with conformal factor $\frac 12 c$ by  Proposition \ref{prop22}. Since $F$ has isotropic S$_{BH}$-curvature, by Theorem 1.1 in \cite{HM} (also Theorem 1.2, \cite{SX}),  $V$ is a homothetic field of $F$ with dilation $\frac 12 c$. The proof is finished.
\end{proof}

\section{Gradient almost Ricci solitons on Randers measure spaces} \label{sec6}

Let $(M, F, m)$ be an $n$-dimensional Randers measure space equipped with a Randers metric $F=\alpha+\beta$ and a smooth measure $m$. As mentioned in the introduction, the volume form determined by $m$ may be written as $dm=e^{-f}dm_{BH}$ and $f$ is the weighted function of $m$.  Theorem \ref{thm11} show that $(M, F, m)$ is a gradient almost Ricci soliton with soliton scalar $\kappa$ if and only if Ric$_\infty=\kappa$, i.e.,
\beq {\rm Ric}(x, y)+\dot S(x, y)=\kappa F^2(x, y), \ \ \ \ (x, y)\in TM. \label{GRS}\eeq

If $(M, F, m)$ is a gradient Ricci soliton (i.e., $\kappa$ is constant), Mo-Zhu-Zhu proved that $F$ must be of isotropic S$_{BH}$-curvature (Proposition 3.1, \cite{MZZ}). From the proof of Proposition 3.1 in \cite{MZZ}, it is easy to see that we need not assume that $F$ is non-Riemannian and $n\geq 2$ as in Theorems \ref{thm12} and \ref{thm13*}. For the gradient almost Ricci soliton $(M, F, m)$, this conclusion still holds by following the proof of Proposition 3.1 in \cite{MZZ}. Thus we have
\begin{prop}\label{prop61} Let $F=\alpha+\beta$ be a Randers metric on an $n$-dimensional measure space $(M, m)$ and $f$ be the weighted function of $m$. If $(M, F, m)$ is a gradient almost Ricci soliton, then $F$ is of isotropic S$_{BH}$-curvature. \end{prop}

Assume that $(M, F, m)$ is a gradient almost Ricci soliton. Then $S_{BH}=(n+1)\sigma F$ for some function $\sigma$ on $M$,  equivalently,  $e_{00}=2\sigma(x)(\alpha^2-\beta^2)$. In this case,  we have (\ref{Xi-formula}). From this and (\ref{Ric-F}), one obtains
\beq {\rm{Ric}}&=&{}^{\alpha}{\rm Ric}+2\alpha s^i_{\ 0;i}-2t_{00}-\alpha^2 t^i_{\ i}\nonumber \\
& & +(n-1)\Big\{s_0^2-2t_0\alpha-4\sigma s_0\alpha-\sigma_0(\alpha-\beta)+s_{0;0}+\sigma^2(3\alpha^2-2\alpha\beta-\beta^2)\Big\}.\label{G-Ric-F} \eeq
On the other hand, the geodesic coefficients $G^i$ of $F$ are given by (see (2.19) in \cite{CS2})
$$G^i=G_\alpha^i+Py^i+Q^i, $$ where $P:=\frac{e_{00}}{2F}-s_0$ and $Q^i:=\alpha s^i_{\ 0}$. Note that $dm=e^{-f}dm_{BH}$. By (\ref{S-curv}),  we have $S=S_{BH}+f_0$ and
\beq \dot S=S_{;0}-2(Py^k+Q^k)S_{y^k}=(S_{BH})_{;0}+f_{;00}-2(Py^k+Q^k)S_{y^k}.\label{S-dot}\eeq
Since $S_{BH}=(n+1)\sigma F$,  by Lemma \ref{lem43}, we get
\beqn & & (S_{BH})_{;0}=(n+1)(\sigma_0 F+\sigma F_{;0})=(n+1)\left\{\sigma_0F+2\sigma^2(\alpha^2-\beta^2)-2\sigma s_0\beta\right\},\\
& & (Py^k+Q^k)S_{y^k}=(n+1)\Big\{\sigma^2(\alpha^2-\beta^2)-\sigma s_0\beta\Big\}+\sigma (\alpha-\beta)f_0-s_0f_0+(f_ks^k_{\ 0})\alpha, \eeqn
where we used $F_{;0}=\beta_{;0}=r_{00}$. Inserting these into (\ref{S-dot}) gives
\beq \dot S=(n+1)\sigma_0F-2\sigma(\alpha-\beta)f_0+2s_0f_0-2(f_ks^k_{\ 0})\alpha+f_{;00}.\label{S-dot-ab}\eeq
It follows from (\ref{GRS}), (\ref{G-Ric-F}) and (\ref{S-dot-ab}) that
\beq \kappa (\alpha^2+2\alpha\beta+\beta^2)&=&{}^{\alpha}{\rm Ric}+2\alpha s^i_{\ 0;i}-2t_{00}-\alpha^2 t^i_{\ i}+2\sigma_0\alpha+2n\sigma_0\beta\nonumber \\
& &+(n-1)\Big\{s_0^2+s_{0;0}-2t_0\alpha-4\sigma s_0\alpha+\sigma^2(3\alpha^2-2\alpha\beta-\beta^2)\Big\}\nonumber \\
& & -2\sigma(\alpha-\beta)f_0+2s_0f_0-2(f_ks^k_{\ 0})\alpha+f_{;00}.\label{abf-kappa} \eeq
 By the same arguments as in (\ref{A1-5-a*})-(\ref{A-135}) or (\ref{hW-xi})-(\ref{h-Ric-V*}), (\ref{abf-kappa}) is rewritten as $A+B\alpha=0$, which means that $A=B=0$ since $\alpha$ is irrational in $y$, where $A$ and $B$ are polynomials in $y$. Based on this,  one obtains from (\ref{abf-kappa}) that
\beq s^i_{\ 0;i}&=&\kappa \beta-\sigma_0+(n-1)(t_0+2\sigma s_0+\sigma^2\beta)+\sigma f_0+f_ks^k_{\ 0},\label{G-si0i}\\
{}^{\alpha}{\rm Ric}&=&\kappa(\alpha^2+\beta^2)+2t_{00}+t^i_{\ i}\alpha^2-2n\sigma_0\beta\nonumber \\
 & &-(n-1)\left(s_0^2+s_{0;0}+3\sigma^2\alpha^2-\sigma^2\beta^2\right)-2(s_0+\sigma\beta)f_0-f_{;00}.\label{G-Ric-a}\eeq

Conversely, assume that  $e_{00}=2\sigma(x)(\alpha^2-\beta^2)$. In the same way as above, we have (\ref{Xi-formula}) and (\ref{G-Ric-F})-(\ref{S-dot-ab}). Moreover, assume that (\ref{G-si0i})-(\ref{G-Ric-a}) hold. Plugging these in (\ref{G-Ric-F}) yields
$${\rm Ric}=\kappa(\alpha+\beta)^2-(n+1)\sigma_0F-2s_0f_0+2\sigma f_0(\alpha-\beta)+2\alpha f_k\mathcal S^k_{\ 0}-f_{:00}.$$
  From this and (\ref{S-dot-ab}), one obtains (\ref{GRS}), that is, $(M, F, m)$ is a gradient Ricci soliton. This proves the following
\begin{thm}\label{thm61} Let $F=\alpha+\beta$ be a Randers metric on an $n$-dimensional measure space $(M, m)$ and $f$ be the weighted function of $m$. Then $(M, F, m)$ is a gradient almost Ricci soliton with soliton scalar $\kappa$ if and only if there is a function $\sigma$ on $M$ such that $e_{00}=2\sigma(\alpha^2-\beta^2)$ and (\ref{G-si0i})-(\ref{G-Ric-a}) holds. \end{thm}

\begin{remark} \label{rm61} {\rm When $\kappa$ is constant, (\ref{G-si0i})-(\ref{G-Ric-a}) are equivalent to (1.4)-(1.5) in \cite{MZZ}. Note that the function $f$ in this paper is different from that in \cite{MZZ} (up to a constant factor $n+1$). Moreover, let $\rho=\ln\sqrt{1-b^2}-f$. By Lemma \ref{lem43}, we calculate $\rho_0$ and $\rho_{0;0}$. Inserting then these in (1.4)-(1.5) in \cite{MZZ} gives (\ref{G-si0i})-(\ref{G-Ric-a}) and vice versa. }\end{remark}

As an application of Theorem \ref{thm61}, we can prove Theorem \ref{thm14}. For this, we need some lemmas. Similar to Lemmas \ref{lem44}-\ref{lem47}, we have
\begin{lem} \label{lem61} Under (\ref{GRS-e00}) and (\ref{GRS-Ric-a}), we have
\beq s^i_{\ 0;i}&=&\kappa \beta+(n-1)(t_0+2\sigma s_0+\sigma^2\beta)+\sigma_0\left\{2(n-1)-(2n-1)b^2\right\}\nonumber \\ & &-\frac{(2n-1)(1-b^2)}{1+b^2}(\sigma_jb^j)\beta-f_{;0j}b^j-\sigma b^2f_0-(f_jb^j)s_0\nonumber \\ & &+\frac 1{1+b^2}(f_{;ij}b^ib^j)\beta-\frac{\sigma(1-b^2)}{1+b^2}(f_kb^k)\beta.\label{GRS-si0i*}\eeq
\end{lem}
\begin{proof} The proof is similar to that of Lemma \ref{lem44}.  In fact, it follows from (\ref{GRS-Ric-a}) and Lemma \ref{lem43} that
\beq {}^{\alpha}{\rm Ric}_{0j}b^j&=&\kappa(1+b^2)\beta+(n+1)t_0-(2n-1)\sigma_0b^2-(\sigma_kb^k)\beta+t^i_{\ i}\beta\nonumber \\
 & &-(n-1)\left\{s_ks^k+\sigma^2(3-b^2)\right\}\beta-\left\{f_{;0j}b^j+(s_0+\sigma\beta)(f_jb^j)+\sigma b^2f_0\right\}, \label{G-Ric-bj}\\
  {}^{\alpha}{\rm Ric}_{jl}b^jb^l&=& \kappa b^2(1+b^2)-(n+1)s_js^j-2n(\sigma_jb^j)b^2+t^i_{\ i}b^2-(n-1)(s_js^j)b^2\nonumber \\
  & &-(n-1)\sigma^2b^2(3-b^2)-f_{;jl}b^jb^l-2\sigma b^2(f_jb^j).\label{G-Ric-bb}\eeq
By putting these and (\ref{rrrr}) in (\ref{sii}), we get
\beqn s^i_{\ ;i}&=&-\kappa b^2-t^j_{\ j}-(n-1)\left(\sigma^2b^2 +2\sigma_jb^j-s_js^j\right)\nonumber \\
& & +\frac{2(2n-1)b^2}{1+b^2}\sigma_jb^j+\frac 1{1+b^2}\left\{f_{;jl}b^jb^l+2\sigma b^2(f_jb^j)\right\}.\eeqn Inserting this in (\ref{ri0i}) gives
\beqn r^i_{\ 0;i}&=& \kappa b^2\beta+t^j_{\ j}\beta+(n-1)\left(\sigma^2b^2+2\sigma_jb^j-s_js^j\right)\beta +2t_0+2\sigma_0(1-b^2)\nonumber \\
  & & -2\sigma(n+1-2b^2)(s_0+2\sigma \beta)-\frac{2(2n-1)b^2}{1+b^2}(\sigma_jb^j)\beta-\frac 1{1+b^2}\left(f_{;jl}b^jb^l+2\sigma b^2(f_jb^j)\right)\beta. \eeqn
 Since $r^i_{\ i;0}=2\sigma_0(n-b^2)-4\sigma(1-b^2)(s_0+2\sigma\beta)$ by Lemma \ref{lem43}, we have
 \beqn r^i_{\ i;0}-r^i_{\ 0;i}&=& -2t_0-(\kappa b^2+t^i_{\ i})\beta +(n-1)\left\{4\sigma^2-\sigma^2b^2+s_is^i-2\sigma_ib^i\right\}\beta\nonumber \\
 & & +2(n-1)\left(\sigma_0+\sigma s_0\right)+\frac{2(2n-1)b^2}{1+b^2}(\sigma_ib^i)\beta +\frac{1}{1+b^2}\left(f_{;jl}b^jb^l+2\sigma b^2(f_jb^j)\right)\beta.\eeqn
Plugging this and (\ref{G-Ric-bj}) into (\ref{si0i**}) yields (\ref{GRS-si0i*}).
\end{proof}

\begin{lem}\label{lem62} Under (\ref{GRS-e00}) and (\ref{GRS-Ric-a}), the equations (\ref{sigma0-f}) and (\ref{G-si0i}) are equivalent. \end{lem}
\begin{proof}  Assume that (\ref{sigma0-f}) holds. Then
\beqn (2n-1)(1-b^2)\sigma_i=\sigma(1+b^2)f_i+f_j(s^j_{\ i}-s^jb_i)+f_{;ij}b^j+(s_i+2\sigma b_i)(f_jb^j), \eeqn which implies that
\beq (2n-1)(1-b^2)\sigma_ib^i=\sigma (1+3b^2)f_ib^i+f_{;ij}b^ib^j-(1+b^2)f_js^j. \label{sigma-jb}\eeq
 Inserting (\ref{sigma0-f}) and (\ref{sigma-jb}) into (\ref{GRS-si0i*}) gives (\ref{G-si0i}).

Conversely, assume that (\ref{G-si0i}) holds. By Lemma \ref{lem61}, the right hand sides of (\ref{G-si0i}) and (\ref{GRS-si0i*}) are equal, which means that
\beq (2n-1)(1-b^2)\sigma_0&=&\sigma(1+b^2) f_0+f_js^j_{\ 0} +f_{;0j}b^j+(f_jb^j)s_0+\frac{(2n-1)(1-b^2)}{1+b^2}(\sigma_jb^j)\beta\nonumber \\
 & & -\frac{1}{1+b^2}(f_{;jl}b^jb^l)\beta+\frac{\sigma(1-b^2)}{1+b^2}(f_jb^j)\beta.\label{sigma-0-f}\eeq
Differentiating this with respect to $y^i$ and then contracting this with $b^i$ yields (\ref{sigma-jb}).  From this and (\ref{sigma-0-f}), one obtains (\ref{sigma0-f}). This finishes the proof.
\end{proof}

 \begin{proof}[Proof of Theorem~{\upshape\ref{thm14}}] It directly follows from Theorem \ref{thm61} and Lemma \ref{lem62}. \end{proof}

\bigskip

Next we give a navigation description for Randers gradient almost Ricci solitons. Let $F=\alpha+\beta$ be a Randers metric expressed by (\ref{F-hW}) in terms of  navigation data $(h, W)$, where $\alpha, \beta$ are given by (\ref{ab}). Recall that $\mathcal R_{ij}$ and $\mathcal S_{ij}$ are defined by (\ref{RS}). Let
\beqn  \mathcal R_j:=W^i\mathcal R_{ij}, \ \  \mathcal R:=W^j\mathcal R_{j},\ \ \mathcal S_j:=W^i\mathcal S_{ij}, \ \  \mathcal S^i_{\ j}:=h^{ik}\mathcal S_{kj}.\eeqn Obviously,  $\mathcal S_iW^i=0$ and $\lambda_{:k}=-2(\mathcal R_k+\mathcal S_k)$.
The geodesic coefficients $G_\alpha^i$ of $\alpha$ are related with those $G_h^i$ of $h$ by
$G^i_\alpha=G^i_h+\zeta^i$, where
 \beqn \zeta^i:=\frac 1{\lambda}(\mathcal R_0+\mathcal S_0)y^i+\frac 12 \mathcal R_{00}W^i+\left(\frac {h^2}{2\lambda}+\frac {W_0^2}{\lambda^2}\right)\left(\mathcal R W^i-\mathcal R^i-\mathcal S^i\right)+\frac {W_0}{\lambda}\left(\mathcal R_0W^i+\mathcal S^i_{\ 0}\right). \label{zeta-i} \eeqn
 From this, we have
 \beq r_{ij}&=&-\mathcal R_{ij}-\left(\frac 1{\lambda}h_{ij}+\frac 2{\lambda^2}W_iW_j\right)\mathcal R+\frac 1{\lambda^2}\left(\mathcal S_iW_j+\mathcal S_jW_i\right)+\frac{1-\lambda}{\lambda^2}\left(\mathcal R_iW_j+\mathcal R_jW_i\right), \label{r-ij-h}\\
 s_{ij}&=&-\frac 1{\lambda}\mathcal S_{ij}+\frac 1{\lambda^2}\left[(\mathcal R_i+\mathcal S_i)W_j-(\mathcal R_j+\mathcal S_j)W_i\right],\label{s-ij-h}\\
 r_j&=&\mathcal R_j+\frac 1{\lambda}\mathcal R W_j-\frac{1-\lambda}{\lambda}(\mathcal R_j+\mathcal S_j), \ \ \  s_j=\mathcal S_j-\frac 1{\lambda}\mathcal RW_j+\frac{1-\lambda}{\lambda}(\mathcal R_j+\mathcal S_j).\label{rs-j-h}\eeq These formulae were first given in \cite{BR} and \cite{Ro}. They also can be found in \cite{ChS} and \cite{CS2}. Recall that $S_{BH}=(n+1)\sigma F$ if and only if
$ \mathcal R_{00}=-2\sigma h^2$ (Lemma \ref{lem51}). In this case, we have
\beq \mathcal R_{ij}=-2\sigma h_{ij}, \ \ \ \ \mathcal R_j=-2\sigma W_j, \ \ \ \ \mathcal R=-2\sigma(1-\lambda)\label{RRRR} \eeq
and
\beq \zeta^i=\frac 1{\lambda}(\mathcal S_0-2\sigma W_0)y^i-\frac{\lambda h^2+2W_0^2}{2\lambda^2}\mathcal S^i+\frac{W_0}{\lambda}\mathcal S^i_{\ 0}.\label{zeta*}\eeq
Moreover, by (\ref{s-ij-h})-(\ref{RRRR}), we get
\beq s_0=\frac 1{\lambda}\mathcal S_0, \ \ \ \ s^i_{\ j}=-\mathcal S^i_{\ j}+\frac 1{\lambda}\mathcal S^iW_j.\label{ss-SS}\eeq
\begin{lem}\label{lem63} Let $F$ be a Randers metric with navigation data $(h, W)$ on a measure space $(M, m)$ and $f$ be the weighted function of $m$. Assume that $S_{BH}=(n+1)\sigma F$. Then
\beqn \dot S(x, y)=(n+1)\sigma_0F-2\sigma f_0 F+2(f_k\mathcal S^k_{\ 0})F+(f_k\mathcal S^k)F^2 +{\rm Hess}_h(f)(y), \label{dot-S-Sk}\eeqn where ${\rm Hess}_h(f)$ stands for the Hessian of $f$ with respect to $h$.
\end{lem}
\begin{proof} By the assumption, we have $G^i_\alpha=G^i_h+\zeta^i$, where $\zeta^i$ are given by (\ref{zeta*}). Thus
\beq {\rm Hess}_\alpha(f)(y)&=&f_{;ij}y^iy^j=f_{:ij}y^iy^j-2f_k\zeta^k \nonumber \\
&=& {\rm Hess}_h(f)(y)-\frac 2{\lambda}f_0(\mathcal S_0-2\sigma W_0)-\frac{2W_0}{\lambda}f_k\mathcal S^k_{\ 0}+\frac{\lambda h^2+2W_0^2}{\lambda^2}f_k\mathcal S^k.\label{Ha-Hh}\eeq
From  (\ref{S-dot-ab}), (\ref{ss-SS})-(\ref{Ha-Hh}), (\ref{hWF}) and $W_0=-{\lambda}\beta$, one obtains that
\beqn \dot S&=&(n+1)\sigma_0F-2\sigma f_0(\alpha-\beta)+2f_0s_0-2(f_ks^k_{\ 0})\alpha+{\rm Hess}_\alpha(f)(y) \nonumber \\
&=&(n+1)\sigma_0F-2\sigma f_0(\alpha-\beta)+\frac 2{\lambda}f_0\mathcal S_0-2f_k\left(-\mathcal S^k_{\ 0}+\frac 1{\lambda}\mathcal S^kW_0\right)\alpha\nonumber \\
& & -\frac 2{\lambda}f_0(\mathcal S_0+2\lambda\sigma\beta)+2\beta f_k\mathcal S^k_{\ 0}+\frac{\lambda h^2+2W_0^2}{\lambda^2}f_k\mathcal S^k+{\rm Hess}_h(f)(y)\nonumber \\
&=& (n+1)\sigma_0F-2\sigma f_0 F+2(f_k\mathcal S^k_{\ 0})F+(f_k\mathcal S^k)F^2+ {\rm Hess}_h(f)(y).\eeqn
This finishes the proof. \end{proof}

 \begin{proof}[Proof of Theorem~{\upshape\ref{thm15}}]  Assume that $(M, F, m)$ is a gradient almost Ricci soliton with soliton scalar $\kappa$, namely, ${\rm Ric}+\dot S=\kappa F^2$.
 Then $S_{BH}=(n+1)\sigma F$ for some function $\sigma=\sigma(x)$ on $M$ by Proposition \ref{prop61}. From Lemma \ref{lem53},  we have
\beq -3(n-1)\sigma_0 F=\widetilde{\rm{Ric}}+\dot S-\tilde\kappa \tilde h^2, \label{tilde-Ric}\eeq  where $\tilde\kappa:=\kappa+(n-1)\sigma^2+2(n-1)\sigma_iW^i$. Recall that $\xi=y-FW$, here $F=F(x, y)=\tilde h(x, \xi)$ by (\ref{tilde-hF}). Then,
\beq & \sigma_0=\tilde\sigma_0+\tilde h\sigma_iW^i, \ \ \ \  f_0=\tilde f_0+\tilde hf_iW^i, \ \ \ \  f_k\mathcal S^k_{\ 0}=f_k\widetilde{\mathcal S}^k_{\ 0}-\tilde hf_k\mathcal S^k, \nonumber\\
& {\rm Hess}_h(f)(y)=f_{:ij}(\xi^i+FW^i)(\xi^j+FW^j)=\tilde f_{:00}+2\tilde h\tilde f_{:0j}W^j +\tilde h^2{\rm Hess}_h(f)(W).\label{Hess-y-xi}\eeq
 From these and Lemma \ref{lem63},  (\ref{tilde-Ric}) becomes
\beqn  & & \widetilde{\rm Ric}+\tilde f_{:00}+\tilde h\left\{2(2n-1)\tilde \sigma_0-2\sigma \tilde f_0+2 f_k\widetilde{\mathcal S}^k_{\ 0}+2\tilde f_{:0j}W^j\right\}\nonumber\\
& &+\tilde h^2\left\{2(2n-1)\sigma_iW^i-2\sigma f_iW^i-f_k\mathcal S^k+{\rm Hess}_h(f)(W)-\tilde \kappa\right\}=0.\eeqn
Since $\tilde h$ is an irrational function in $\xi$, by the same arguments as in (\ref{abf-kappa})-(\ref{G-Ric-a}), this equation is equivalent to the following two equations:
\beq & (2n-1)\tilde\sigma_0 -\sigma \tilde f_0+ f_k\widetilde{\mathcal S}^k_{\ 0}+\tilde f_{:0j}W^j=0, \label{sfSWx} \\
& \widetilde{\rm Ric}+\tilde f_{:00}+\tilde h^2\Big\{2(2n-1)\sigma_iW^i-2\sigma f_iW^i-f_k\mathcal S^k+{\rm Hess}_h(f)(W)-\tilde \kappa\Big\}=0.\label{Ric-xi}\eeq
Obviously, (\ref{sfSWx}) is equivalent to (\ref{sfSWx*}) and
%  which implies that $\sigma$ is constant when $f$ satisfies   \beqn \sigma  f_i-f_k{\mathcal S}^k_{\ i}- f_{:ij}W^j=0. \label{sigma-cons}\eeqn
 (\ref{Ric-xi}) implies that there is a function $\mu=\mu(x)$ on $M$ such that
\beq {}^h{\rm Ric}_{ij}\xi^i\xi^j+f_{:ij}\xi^i\xi^j=\mu \tilde h^2, \ \ \ {\rm equivalently}\ \ \ {}^h{\rm Ric}_{ij}+f_{:ij}=\mu h_{ij}, \label{h-Ric-f}\eeq which means that $(M, h, f)$ is a Riemannian gradient almost Ricci soliton with soliton scalar $\mu$, i.e.,
\beqn {}^h{\rm Ric}(y)+{\rm Hess}_h(f)(y)=\mu h^2. \eeqn Thus (\ref{Ric-xi}) is reduced to
\beq 2(2n-1)\sigma_iW^i- 2\sigma f_iW^i-f_k\mathcal S^k+{\rm Hess}_h(f)(W)=\tilde \kappa-\mu.\label{fW}\eeq
Differentiating (\ref{sfSWx}) in $\xi^i$ and then contracting this with $W^i$ yield
 \beq {\rm Hess}_h(f)(W)=\sigma f_iW^i+f_i\mathcal S^i-(2n-1)\sigma_iW^i.\label{fwi}\eeq
Plugging this in  (\ref{fW}) and using  $\tilde\kappa=\kappa+(n-1)\sigma^2+2(n-1)\sigma_iW^i$ lead to (\ref{H-hfW}).

 Conversely, assume that  $(M, h, f)$ is a Riemannian gradient almost Ricci soliton with  soliton scalar $\mu$, $W$ is a conformal vector field of $h$ with conformal factor $-\sigma$, equivalently, $S_{BH}=(n+1)\sigma F$, and $f$ satisfies (\ref{sfSWx*})-(\ref{H-hfW}). Thus we get (\ref{h-Ric-f}) and (\ref{fwi}) from (\ref{sfSWx*}). Moreover, by (\ref{sfSWx*}) and $y^i=\xi^i+FW^i$, we have
 \beq  \sigma f_0-f_k\mathcal S^k_{\ 0}=(2n-1)\sigma_0+f_{:ij}\xi^iW^j+F{\rm Hess}_h(f)(W). \label{fwyi}\eeq From (\ref{Hess-y-xi}) and (\ref{fwi}) we get
\beq {\rm Hess}_h(f)(y)= f_{:ij}\xi^i\xi^j+2F f_{:ij}\xi^iW^j +F^2\Big\{\sigma f_iW^i+f_i\mathcal S^i-(2n-1)\sigma_iW^i\Big\}. \label{Hfy}\eeq
Thus, it follows from Lemma \ref{lem63}, (\ref{fwyi})-(\ref{Hfy}), (\ref{h-Ric-f}), (\ref{fwi}) and (\ref{H-hfW}) that
\beqn \widetilde{\rm Ric}+\dot S&=&\widetilde{\rm Ric}+(n+1)\sigma_0 F-2F\Big\{(2n-1)\sigma_0+f_{:ij}\xi^iW^j+F{\rm Hess}_h(f)(W)\Big\}\nonumber \\
& &+(f_i\mathcal S^i)F^2+ f_{:ij}\xi^i\xi^j+2F f_{:ij}\xi^iW^j +F^2\Big\{\sigma f_iW^i+f_i\mathcal S^i-(2n-1)\sigma_iW^i\Big\}\nonumber \\
&=& \mu \tilde h^2-3(n-1)\sigma_0 F+\tilde h^2\Big\{-2{\rm Hess}_h(f)(W)+\sigma f_iW^i+ 2f_i\mathcal S^i-(2n-1)\sigma_iW^i\Big\}\nonumber \\
&=&\mu \tilde h^2-3(n-1)\sigma_0 F+ \tilde h^2\Big\{-\sigma f_iW^i+\sigma_iW^i+2(n-1)\sigma_iW^i\Big\}\nonumber \\
&=&\mu \tilde h^2-3(n-1)\sigma_0 F+\Big\{\kappa-\mu+(n-1)\sigma^2+ 2(n-1)\sigma_iW^i\Big\}\tilde h^2\nonumber \\
&=& \tilde \kappa \tilde h^2-3(n-1)\sigma_0 F, \eeqn that is, $-3(n-1)\sigma_0 F=\widetilde{\rm Ric}+\dot S -\tilde\kappa\tilde h^2.$ From this and (\ref{Ric-Ric}), one obtains that ${\rm Ric}+\dot S=\kappa F^2$, i.e., $(M, F, m)$ is a gradient almost Ricci soliton.
\end{proof}

\section{Examples for nontrivial gradient Ricci solitons} \label{sec7}

In this section, we give some examples of nontrivial shrinking, steady and expanding gradient Ricci solitons based on Corollary \ref{cor16}.
By Theorem \ref{thm15} (especially, Corollary \ref{cor16}),  the Riemannian structure $(M, h, f)$ associated to a Randers gradient almost Ricci soliton with navigation data $(h, W)$ shall be a gradient almost Ricci soliton as well.

\begin{ex} \label{ex71} {\bf (Gaussian solitons)}  {\rm Let $(\mathbb R^n, h)$ denote a Euclidean space with its standard metric $h=|\cdot|$.  First, one may take $\kappa=0$ and $V=0$, and see $(\mathbb R^n, h, V)$ as a steady Ricci soliton (trivial soliton). However, for any $\rho\in \mathbb R$, if we take $f=\frac \rho 2|x|^2$ for any $x\in \mathbb R^n$, then $(\mathbb R^n, h, f)$ is a Riemannian expanding or shrinking gradient Ricci soliton with soliton constant $\rho$ depending on the signs of $\rho$. The choice of $f$ is unique up to a constant. The soliton $(\mathbb R^n, |\cdot|, e^{-\frac \rho 2|x|^2})$ is called the {\it Gaussian soliton} (\cite{CK}, \cite{CLN}).

Let $m$ be the measure determined by $dm=e^{-f}dm_{BH}$ on $\mathbb R^n$, where $f=\frac \rho 2|x|^2$. Assume that $F$ is of constant S$_{BH}$-curvature $\sigma$, equivalently, $W_{i:j}+W_{j:i}=-4\sigma \delta_{ij}$. Solving this equation yields
 $W=-2\sigma x+Qx+C$, where $Q=(q^i_{\ j})$ is a skew-symmetric matrix and $C$ is a constant vector in $\mathbb R^n$ (cf. \cite{BRS}). Obviously, $\mathcal S^i_{\ j}=\mathcal S_{ij}=q_{ij}$. On the other hand, $(f, W)$ satisfies (\ref{f-sigma-hw}), i.e., $\sigma f_0-f_k\mathcal S^k_{\ 0}-f_{:0j}W^j=0$ if and only if either $\rho=0$,  or $\sigma=0$, $C=0$ by solving PDE.  This means that $W=Qx$ when $\rho\neq 0$ and $W=-2\sigma x+Qx+C$ when $\rho=0$.
  If $\rho=0$, then $f\equiv 1$ and $\dot S=0$ by Lemma \ref{lem63}. In this case, the gradient Ricci soliton is trivial. When $\rho\neq 0$, we have $W=Qx=0$ since $\|W\|_h<1$ for any $x\in \mathbb R^n$. Thus, the Randers metric $F$ expressed by (\ref{F-hW}) must be Euclidean, i.e., $F(x, y)=h(y)=|y|$ for any $y\in \mathbb R^n$. By Corollary \ref{cor16}, there is no non-Riemannian Randers gradient Ricci soliton in $\mathbb R^n$ with respect to the measure $e^{-\frac{\rho}2|x|^2}dx$ ($\rho\neq 0$). However, if we consider the subset $\Omega:=\{x\in \mathbb R^n\|W\|_h<1\}$ in $\mathbb R^n$, the metric $\bar F$ expressed by (\ref{F-hW}) in which $W=Qx$ is a Randers metric on $\Omega$. In this case, $(\Omega, \bar F, m)$ is a gradient Ricci soliton with soliton constant $\rho$ with respect to the measure $m$ determined by $dm=e^{-\frac{\rho}2|x|^2}dx$ by Corollary \ref{cor16}. }\end{ex}

\begin{ex} \label{ex72} {\bf (Steady Ricci soliton)} {\rm  Let $(\mathbb R^2, h)$ be a complete noncompact Riemannian manifold with complete metric $h=\frac{dx^2+dy^2}{1+x^2+y^2}$. In polar coordinates $(r, \theta)$, $h$ is rewritten as
\beq h=\frac{dr^2+r^2d\theta^2}{1+r^2}.\label{h-R2}\eeq  Let $t=\operatorname{arcsinh}r=\log(r+\sqrt{1+r^2})$. Then
\beq h=dt^2+\tanh^2td\theta^2, \ \ \ \ \ ^h{\rm Ric}(Y)=\frac 2{\cosh^2t}h^2,\label{h-Ric-R2}\eeq where $h^2=h(Y, Y)$ for $Y=y^1\frac{\partial}{\partial t}+y^2\frac{\partial}{\partial \theta}\in T_{(t, \theta)}\mathbb R^2$. The Christoffel symbols of $h$ are given by  \beq \Gamma_{11}^1=\Gamma_{11}^2=\Gamma_{12}^1=\Gamma_{22}^2=0, \ \ \ \Gamma_{12}^2=\frac 2{\sinh(2t)}, \ \ \ \Gamma_{22}^1=-\frac{\tanh t}{\cosh^2t}.\label{Gamma-R2}\eeq
Let \beq f(t, \theta):=-2\log(\cosh t).\label{f-R2}\eeq  Then \beq & f_1=f'(t)=-2\tanh t, \ \ \ f_2=f'(\theta)=0, \label{fi-R2} \\
 & f_{:11}=f''(t)=-\frac 2{\cosh^2 t}, \ \ \ f_{:12}=f_{:21}=0, \ \ \ f_{:22}=-\frac{2\tanh^2t}{\cosh^2t}, \label{fij-R2}\eeq
 which implies that Hess$_h(f)(Y)=-\frac 2{\cosh^2 t}h^2.$  Consequently, $^h{\rm Ric}(Y)+{\rm Hess}_h(f)(Y)=0$, namely,  $(\mathbb R^2, h, f)$ is a Riemannian gradient steady Ricci soliton, called the {\it Hamilton's cigar soliton} (\cite{Ha3}). By Lemma 2.7 in \cite{CK},  $(\mathbb R^2, h, f)$ defined by (\ref{h-R2}) and (\ref{f-R2}) is the unique rotationally symmetric gradient Ricci soliton with  positive sectional curvature on $\mathbb R^2$ up to homothety.

Assume that $F$ has constant S$_{BH}$-curvature $\sigma$, i.e.,  $S_{BH}=3\sigma F$, and  $W=W^1\frac{\partial}{\partial t}+W^2\frac{\partial}{\partial \theta}$ is a vector field on $\mathbb R^2$.  Consider the following PDE system:
\beq W_{i:j}+W_{j:i}=-4\sigma h_{ij}, \ \ \ \  \sigma  f_0-f_k{\mathcal S}^k_{\ 0}- f_{:0j}W^j=0. \label{WFS} \eeq
From (\ref{h-R2})-(\ref{fij-R2}), (\ref{WFS})$_2$ can be simplified as
\beqn \sigma f_1-f_{:11}W^1=0, \ \ \ \ \sigma f_2-f_1\mathcal S^1_{\ 2}-f_{:22}W^2=0, \eeqn which implies that
\beq W^1=W_1=\frac 12 \sigma\sinh(2t), \ \ \ (W_1)_\theta-(W_2)_t=-\frac{2\tanh t}{\cosh^2 t}W^2=-\frac 4{\sinh(2t)}W_2,\label{WW-12}\eeq
 and (\ref{WFS})$_1$ is simplified as \beq (W_1)_t=-2\sigma, \ \ \  (W_1)_\theta+(W_2)_t=\frac 4{\sinh(2t)}W_2, \ \ \  (W_2)_{\theta}+\frac{\tanh t}{\cosh^2 t}W_1=-2\sigma \tanh^2t.\label{WW-12-tt}\eeq From (\ref{WW-12})$_1$ and (\ref{WW-12-tt})$_1$, we get $\sigma=0$ and hence $W^1=W_1=0$. From (\ref{WW-12-tt})$_3$, we have $(W_2)_\theta=0$, i.e., $W_2$ is only a function of $t$. Thus (\ref{WW-12})$_2$ and (\ref{WW-12-tt})$_2$ are reduced to $(\log W_2)_t=\frac 4{\sinh(2t)}$. Integrating this yields $W_2=C\tanh^2t$ for some constant $C>0$. Up to a positive constant, $W_2=\tanh^2t$ and hence $W^2=1$. Thus (\ref{WFS}) has the unique solution $W=\frac{\partial}{\partial \theta}$. In this case, $\lambda=1-\|W\|_h^2=\frac 1{\cosh^2t}>0$ and $W_0=(\tanh^2t)y^2$.

Let $F$ be expressed by (\ref{F-hW}),  that is,
$$F(t, \theta; Y)=(\cosh t)\sqrt{(y^1)^2+\sinh^2t(y^2)^2}-(\sinh^2t)y^2, $$ where $Y=y^1\frac{\partial}{\partial t}+y^2\frac{\partial}{\partial \theta}\in \mathbb R^2$, and
$m$ be the measure determined by $dm=e^{-f}dm_{BH}$ on $(M, F)$.
Back to the original coordinate system $(x, y)$,  $f=-2\log\sqrt{1+x^2+y^2}$ and
\beq F(\hat x, Y)=\sqrt{\frac {(ux+vy)^2+(1+x^2+y^2)(uy-vx)^2}{x^2+y^2}}+uy-vx. \label{F-R2}\eeq where $\hat x=(x, y)\in \mathbb R^2$ and $Y=(u, v)\in T_{\hat x}\mathbb R^2\cong \mathbb R^2$.  Observe that the point $(x, y)=(0, 0)$ is a removable singular point of $F$. Indeed, $\lim\limits_{\hat x\rightarrow 0} F(\hat x, Y)=|Y|=\sqrt{u^2+v^2}$. Thus $F$ is a regular Randers metric on $\mathbb R^2$.
 % we have $\cosh t=\sqrt{1+x^2+y^2}$ and
%$\tanh^2 t=\frac {r^2}{1+r^2}=\frac{x^2+y^2}{1+x^2+y^2}$. Moreover,
% \beqn \xi=(\xi^1, \xi^2)\left(\begin{array}{ll}(\cosh t)\frac{\partial}{\partial r}\\ \ \ \ \frac{\partial}{\partial \theta}\end{array}\right)=(\xi^1, \xi^2)\left(\begin{array}{ll} \cosh t\cos\theta & \cosh t\sin\theta \\ -r\sin\theta & r\cos\theta\end{array}\right) \left(\begin{array}{ll}\frac{\partial}{\partial x}\\ \frac{\partial}{\partial y}\end{array}\right)\eeqn
%Denote $\xi=u\frac{\partial}{\partial x}+v\frac{\partial}{\partial y}$. Then
%\beqn  (\xi^1, \xi^2)=\frac 1{r\cosh t}(u, v)\left(\begin{array}{ll} r\cos\theta & -\cosh t\sin\theta \\ r\sin\theta & \cosh t\cos\theta \end{array}\right).\eeqn
By Corollary \ref{cor16}, $(\mathbb R^2, F, m)$ is a Randers gradient steady Ricci soliton, called a {\it Randers cigar Ricci soliton}.
Obviously, it is invariant under all rotations in $\mathbb R^2$, i.e., $F(A\hat x, AY)=F(\hat x, Y)$ for all $A\in O(2)$, that is,  it is a rotationally symmetric metric, called a {\it spherically symmetric metric} introduced by L. Zhou (\cite{Zh}).
By Lemma \ref{lem53} and (\ref{h-Ric-R2}), we have ${\rm Ric}=\frac 2{\cosh^2t} F^2$, which implies that $F$ is of positive flag curvature $\mathbf K_F=\frac 2{\cosh^2t}$.

Note that every Randers metric may be obtained in terms of the navigation data $(h, W)$. As mentioned before,  $(\mathbb R^2, h, f)$ is the unique rotationally symmetric gradient Ricci soliton with positive sectional curvature on $\mathbb R^2$ up to a homothety and $W=\frac{\partial}{\partial \theta}$ is a unique solution of (\ref{WFS}). Then $(\mathbb R^2, F, m)$ is a unique spherically symmetric gradient Ricci soliton. Summing up, one obtains
\begin{prop} The Randers cigar soliton $(\mathbb R^2, F, m)$ defined by (\ref{F-R2}) and (\ref{f-R2}) is the unique gradient Ricci soliton of constant S$_{BH}$-curvature and positive flag curvature on $\mathbb R^2$  up to homothety. In this case, S$_{BH}\equiv 0$ and $\mathbf K_F=\frac 2{\cosh^2t}$. \end{prop}
}\end{ex}

% \begin{ex} \label{ex73} Consider a rotationally symmetric Riemannian metric $h$ on $\mathbb R^2$ of the form
% \beqn h=\tau^2(r)dr^2+r^2d\theta^2, \eeqn

\begin{ex} \label{ex73} {\bf (Shrinking Ricci soliton)} {\rm Let $(\mathbb S^{2m-1}, \hat h)$ be a $2m-1 (\geq 3)$-dimensional sphere in $\mathbb R^{2m}$ equipped with the standard Riemannian metric $\hat h$ of positive constant curvature $\mu$. For any $p\in \mathbb S^{2m-1}$, we identify $T_p\mathbb S^{2m-1}$ with $\mathbb R^{2m-1}$ in a natural way. Let $\mathbb S^{2m-1}_+$ and $\mathbb S^{2m-1}_-$ be the upper and lower hemisphere of $\mathbb S^{2m-1}$ and let
$$\psi_{\pm}: T_p\mathbb S^{2m-1}\cong \mathbb  R^{2m-1}\rightarrow \mathbb S_{\pm}^{2m-1}, \ \ \ \psi_{\pm}(x):=\left(\frac{x}{\sqrt{1+\mu|x|^2}}, \frac{\pm 1}{\sqrt{\mu(1+\mu|x|^2)}}\right),$$
be the stereographic projection from the center respectively. $\psi_{\pm}$ sends straight lines in $\mathbb R^{2m-1}$ to great circles on $\mathbb S^{2m-1}_{\pm}$. Denote by $x=(x^i)\in \mathbb S^{2m-1}$, where $\{x^i\}$ are usually called the {\it projective coordinates} on $\mathbb S^{2m-1}$. The metric $\hat h$ on $\mathbb S^{2m-1}$ can be expressed in projective form:
\beq \hat h=\frac {\sqrt{(1+\mu|x|^2)|y|^2-\mu\langle x, y\rangle ^2}}{1+\mu|x|^2}, \label{hat-h}\eeq
where $\langle x, y\rangle=\sum_ix^iy^i$ for any $y=y^i\frac{\partial}{\partial x^i}\in T_x\mathbb S^{2m-1}\cong \mathbb R^{2m-1}$. From this, we have
\beq \hat h_{ij}=\frac{\delta_{ij}}{1+\mu|x|^2}-\frac{\mu x_ix_j}{(1+\mu|x|^2)^2}, \ \ \ \ \hat h^{ij}=(1+\mu|x|^2)(\delta^{ij}+\mu x^ix^j),\label{hij}\eeq The Christoffel symbols of $\hat h$ are given by
$$\hat \Gamma_{ij}^k=-\frac {\mu}{1+\mu|x|^2}(x_i\delta^k_j+x_j\delta^k_i)$$ and  $^{\hat h}{\rm Ric}=2(m-1)\mu \hat h^2$.

 Let $Q=(q_{ij})$ be an skew-symmetric matrix and $d=(d^i)$ be a constant vector in $\mathbb R^{2m-1}$ with $|d|<1$,  $Q^TQ+\mu dd^T=\mu|d|^2E$ ($E$ is an identity matrix) and $Qd=0$. For example, in $3$-dimensional case,  we can take
\beq Q =\sqrt{\mu}\left(\begin{array}{ccc} 0 & p & q  \\  -p & 0&l \\ -q &-l &0 \end{array}\right), \ \ \ d=\pm \left(\begin{array}{c} l  \\  -q \\ p
\end{array}\right)\neq 0, \label{Q1} \eeq
where $p, q, l$ are  constants with $|d|^2=p^2+q^2+l^2<1$. It is easy to check that $Q$ and $d$ satisfy $Q^TQ+\mu dd^T=\mu |d|^2 E$ and $Qd=0$.
Then
 \beq \widehat{W}:= Qx+\mu\langle x , d\rangle x+d\label{W-Qd-1}\eeq
is a globally defined Killing vector field on the open upper hemisphere $(\mathbb S_+^{2m-1}, \hat h)$ with $\|\widehat W\|_{\hat h}=|d|<1$. Note that the expression (\ref{W-Qd-1}) should be understood in projective coordinates. Indeed, a direct calculation shows that $\psi_{+*}(\widehat W)=A p_+,$  where
 $$A:=\left(\begin{array}{ll} \ \ \ \ Q & \sqrt{\mu}d\\-\sqrt{\mu}d& \ \ \ \ 0\end{array}\right), \ \ \ \ \
 p_+:=\left(\begin{array}{ll}\frac{x}{\sqrt{1+\mu|x|^2}}\\ \frac 1{\sqrt{\mu(1+\mu|x|^2)}}\end{array}\right)=\psi_+(x)\in \mathbb S_+^{2m-1}.$$
  The continuity of $\widehat W$ on the closed hemisphere implies that its value at any point $p$ on the equator is also the matrix product $A p$. We extend $\widehat W$ to the open lower hemisphere by insisting that $\psi_{-*}(\widehat W)=A p_-$, where $$p_-:= \left(\begin{array}{ll}\frac{x}{\sqrt{1+\mu|x|^2}}\\ \frac {-1}{\sqrt{\mu(1+\mu|x|^2)}}\end{array}\right)=\psi_{-}(x)\in \mathbb S_-^{2m-1}.$$ It follows that $\widehat W=Qx-\mu\langle x , d\rangle x-d$. Obviously,
  the points $p_{\pm}$ goes to the poles $(0, \pm \frac 1{\sqrt{\mu}})^T$ as $x\rightarrow 0$. We construct $\widehat W$ from the
data $(Q, d)$ on the upper hemisphere and from $(Q, -d)$ on the lower hemisphere. The actual Killing field on $\mathbb S^{2m-1}$ has the value $A p$ at any point $p$, including the equator. Since $A$ is a constant matrix, the constructed $\widehat W$ is globally defined and smooth. Moreover, $\|\widehat W\|_{\hat h}=|d|<1$.

Let $M:=\mathbb R\times \mathbb S^{2m-1}$, which is an $2m (m\geq 2)$-dimensional cylinder equipped with a standard product metric $h^2=dt^2+\hat h^2$,  where $t\in \mathbb R$.
 For any $y=(y^1, \hat y)\in T_{(t, x)}M=\mathbb R\oplus T_x\mathbb S^{2m-1}$, where $(t, x)\in \mathbb R\times \mathbb S^{2m-1}$ and $\hat y=\sum_{i=2}^{2m}y^i\frac{\partial}{\partial x^i}\in T_x\mathbb S^{2m-1}$. Then $h^2(y)=(y^1)^2+\hat h^2(\hat y)$. By a direct calculation, the Christoffel symbols of $h$ are given by
\beq \Gamma_{11}^1=\Gamma_{11}^k=\Gamma_{1j}^1=\Gamma_{1j}^k=\Gamma_{ij}^1=0, \ \ \ \ \Gamma_{ij}^k=\hat \Gamma_{ij}^k= -\frac {\mu}{1+\mu|x|^2}(x_i\delta^k_j+x_j\delta^k_i), \label{gamma-1}\eeq where $2\leq i, j, k\leq 2m$, and $^h{\rm Ric}=2(m-1)\mu \hat h^2$.
Now we choose a function $$f(t, x):=(m-1)\mu t^2$$ and a vector field $W:=\widehat W$ on $M$, i.e., the component $W^1$ of $W$ is zero in $t$-direction and the rest components $W^i$ coincide with $\widehat W^i$, equivalently, $W_1=0$ and $W_i=\widehat W_i=\sum_{j=2}^{2m}\hat h_{ij}W^j$, where $2\leq i\leq 2m$. In this case, $\|W\|_h=\|\widehat W\|_{\hat h}=|d|<1$ and
$$W_0=\widehat{W}_i\hat y^i=\frac{\langle Qx+d , \hat y\rangle}{1+\mu|x|^2}.$$
It is easy to check from (\ref{gamma-1}) that
\beq f_{:11}=2(m-1)\mu, \ \ \ f_{:1j}=f_{:j1}=f_{:jk}=0,\label{ff-1}\eeq which mean that
Hess$_h(f)(y)=2(m-1) \mu(y^1)^2$,  and
\beq W_{1:1}=W_{1:j}=W_{j:1}=0, \ \ \ \ W_{i:j}=\frac{q_{ij}}{1+\mu|x|^2}+\frac{\mu\left(q_{jk}x^k x^i+x^id^j-q_{ik}x^kx^j-x^jd^i\right)}{(1+\mu|x|^2)^2},\label{WW-1}\eeq which mean that $W$ is a Killing vector field of $h$ on $M$, where $2\leq i, j, k\leq 2m$.  Thus,  $$^h{\rm Ric}+{\rm Hess}_h(f)(y)=2(m-1)\mu h^2,$$ that is, $(M, h, f)$ is a shrinking gradient Ricci soliton with soliton constant $2(m-1)\mu$.  Moreover, by (\ref{ff-1})-(\ref{WW-1}),  $f$ satisfies $f_k\mathcal S^k_{\ 0}+f_{:0k}W^k=0$ ($1\leq k\leq 2m$).

Let $F$ be defined by (\ref{F-hW}) in terms of $(h, W)$, i.e.,
\beqn F(x, y)&=&\frac{\sqrt{\lambda(1+\mu|x|^2)^2((y^1)^2+\hat h^2)+\langle Qx+d, \hat y\rangle^2}}{\lambda(1+\mu|x|^2)}-\frac{\langle Qx+d, \hat y\rangle}{\lambda(1+\mu|x|^2)}, \nonumber \\
&=&\frac{\sqrt{\lambda(1+\mu|x|^2)\left\{(1+\mu|x|^2)(y^1)^2+|\hat y|^2\right\}-\lambda\mu\langle x, \hat y\rangle^2+\langle Qx+d, \hat y\rangle^2}}{\lambda(1+\mu|x|^2)}-\frac{\langle Qx+d, \hat y\rangle}{\lambda(1+\mu|x|^2)}
\eeqn where $\lambda=1-|d|^2$, $x\in \mathbb R^{2m-1}$ and $y=(y^1, \hat y)\in T_{(t, x)}M$,  and $m$ be the measure determined by $dm=e^{-(m-1)\mu t^2}dm_{BH}$. By Corollary \ref{cor16}, $(M, F, m)$ is a shrinking gradient Ricci soliton with soliton constant $2(m-1)\mu$. }\end{ex}

\begin{ex} {\bf (Expanding Ricci soliton)}\label{ex74} {\rm Let $(\mathbb S^{2m-1}, \hat h)$ be a $2m-1 (\geq 3)$-dimensional sphere in $\mathbb R^{2m}$ equipped with the standard Riemannian metric $\hat h$ of positive constant curvature $\mu$ as in Example \ref{ex73}. Also, we choose the vector field $\widehat W$ as in Example \ref{ex73}, which induces a global smooth Killing vector field on $\mathbb S^{2m-1}$.

 Let $M:=(0, 1)\times \mathbb S^{2m-1}$ be an $2m$ $(m\geq 2)$-dimensional cylinder equipped with a warped product metric $h^2=dt^2+t^2\hat h^2$,  where $t\in (0, 1)$. Then $h^2(y)=(y^1)^2+t^2\hat h^2(\hat y)$ for any $y=(y^1, \hat y)\in T_{(t, x)}M=T_t(0, 1)\oplus T_x\mathbb S^{2m-1}$. By a direct calculation, the Christoffel symbols of $h$ are given by
$$\Gamma_{11}^1=\Gamma_{1j}^1=\Gamma_{11}^k=\Gamma_{1j}^1=0, \ \ \ \Gamma_{1j}^k=t^{-1}\delta^k_j, \ \ \ \Gamma_{ij}^1=-t\hat h_{ij}, \ \  \Gamma_{ij}^k= -\frac {\mu}{1+\mu|x|^2}(x_i\delta^k_j+x_j\delta^k_i),$$ where $2\leq i, j, k\leq 2m$, and $^h{\rm Ric}=2(m-1)(\mu-1) \hat h^2$. In particular, when $\mu=1$, $^h{\rm Ric}=0$.
Now we choose a function $$f(t, x):=-(m-1)\mu t^2$$ and a vector field $W:=\widehat W$ on $M$, i.e., the component $W^1$ of $W$ is zero in $t$-direction and the rest components $W^i$ coincide with $\widehat W^i$, equivalently, $W_1=0, W_i=t^2\widehat W_j$. In this case, $\|W\|_h=t\|\widehat W\|_{\hat h}=t|d|<1$ and
 \beqn W_0=t^2\widehat W_0= \frac{t^2\langle Qx+d , \hat y\rangle}{1+\mu|x|^2}.\eeqn  It is easy to check that
$$f_{:11}=-2(m-1)\mu, \ \ f_{:1j}=f_{:j1}=0, \ \ f_{:jk}=-2(m-1)\mu t^2\hat h_{jk}, $$
which mean that Hess$_h(f)(y)=-2(m-1)\mu h^2$,  and
 $$W_{1:1}=0, \ \ \ W_{1:j}=-W_{j:1}=- t\widehat W_j, \ \ \  W_{i:j}=\frac{q_{ij}}{1+\mu|x|^2}+\frac{\mu t^2\left(q_{jk}x^k x^i+x^id^j-q_{ik}x^kx^j-x^jd^i\right)}{(1+\mu|x|^2)^2},$$ which mean that $W$ is a Killing vector field of $h$ on $M$, where $2\leq i, j, k\leq 2m$. Thus $$^h{\rm Ric}+{\rm Hess}_h(f)(y)=-2(m-1)\mu h^2.$$ Hence, $(M, h, f)$ is a Riemannian expanding gradient Ricci soliton with soliton constant $-2(m-1)$ when $\mu=1$.  Note that $W^1=f_j=0$,  $\mathcal S_{11}=0$ and  $\mathcal S_{1j}=-t\widehat W_j$ ($2\leq j\leq 2m$). Consequently, $f$ satisfies
 \beqn f_k\mathcal S^k_{\ 0}+f_{:0k}W^k&=&\sum_{j=2}^{2m}f_1\mathcal S_{1j}y^j+\sum_{j=2}^{2m}f_{:1j}W^jy^1+\sum_{j, l=2}^{2m}f_{:lj}W^jy^l \\ &=& 2(m-1)\mu t^2\sum_{j=2}^{2m}\widehat W_jy^j-2(m-1)\mu t^2\sum_{j, l=2}^{2m}\hat h_{lj}\widehat W^jy^l=0,\eeqn where $1\leq k\leq 2m$.

 Assume that $\mu=1$. Let $F$ be defined by (\ref{F-hW}) in terms of $(h, W)$, i.e.,
\beqn F(x, y)
%&=&\frac{\sqrt{\lambda(1+\mu|x|^2)^2((y^1)^2+t^2\hat h^2)+t^4\langle Qx+d, \hat y\rangle^2}}{\lambda(1+\mu|x|^2)}-\frac{t^2\langle Qx+d, \hat y\rangle}{\lambda(1+\mu|x|^2)}, \nonumber \\
&=&\frac{\sqrt{\lambda(1+|x|^2)\left\{(1+|x|^2)(y^1)^2+t^2|\hat y|^2\right\}-\lambda t^2\langle x, \hat y\rangle^2+t^4\langle Qx+d, \hat y\rangle^2}}{\lambda(1+|x|^2)} \\
& &-\frac{t^2\langle Qx+d, \hat y\rangle}{\lambda(1+|x|^2)}
\eeqn where $\lambda=1-t^2|d|^2>0$, $t\in (0, 1)$, $x\in \mathbb R^{2m-1}$ and $y=(y^1, \hat y)\in T_{(t, x)}M$,  and $m$ be the measure determined by $dm=e^{(m-1) t^2}dm_{BH}$. By Corollary \ref{cor16}, $(M, F, m)$ is an expanding gradient Ricci soliton with soliton constant $-2(m-1)$.
% In fact, the previous construction also works on $\mathbb R^k\times \mathbb S^{2m-k}$, where $k$ is an odd number.
 }\end{ex}

It is worth mentioning that the constructions in Examples \ref{ex73}-\ref{ex74} do not work on $M=\mathbb R\times \mathbb S^{2m}$ because $\mathbb S^{2m}$ does not admit a nowhere zero Killing vector field by Berger's Theorem (Theorem 38, \cite{Pet}). Note that any Riemannian space with negative Ricci curvature does not admit a nowhere zero Killing vector field of constant length (\cite{BN}). The above construction can not be applied to $M=\mathbb R\times \mathbb H^{n}$, where $\mathbb H^n$ is an $n$-dimensional hyperbolic space of constant negative curvature.

\bigskip

% {\bf Data Availability}  Data sharing is not applicable to this article as no data sets were generated or analyzed during the current study.

{\bf{ACKNOWLEDGMENT.}}
The author is very grateful to the anonymous referee for his/her careful reading and many valuable suggestions or comments on the manuscript. These have greatly improved the presentation of the paper.

 \end{document}